\documentclass[11pt]{amsart}
\usepackage{mathrsfs}
\usepackage{amssymb}

\usepackage{amscd}
\usepackage{amsmath, amssymb}
\usepackage{amsfonts}
\usepackage[colorlinks,linkcolor=blue,citecolor=blue, pdfstartview=FitH]{hyperref}

\usepackage{amscd}
\usepackage{easybmat}
\usepackage{mathrsfs}
\usepackage{amsfonts}
\usepackage{color}
\usepackage{pifont}
\usepackage{upgreek}
\usepackage{bm}
\usepackage{hyperref}
\usepackage{shorttoc}
\usepackage{amsmath,amstext,amsthm,a4,amssymb,amscd}
\usepackage[mathscr]{eucal}
\usepackage{mathrsfs}
\usepackage{epsf}

\numberwithin{equation}{section}

\def\p{\partial}
\def\o{\overline}
\def\b{\bar}
\def\mb{\mathbb}
\def\mc{\mathcal}
\def\n{\nabla}

\theoremstyle{plain}
\newtheorem{thm}{Theorem}[section]
\newtheorem{lemma}[thm]{Lemma}
\newtheorem{prop}[thm]{Proposition}
\newtheorem{cor}[thm]{Corollary}
\theoremstyle{definition}
\newtheorem{rem}[thm]{Remark}

\theoremstyle{definition}
\newtheorem{defn}[thm]{Definition}
\newcommand{\comment}[1]{}

\usepackage{fancyhdr}
\begin{document}

\title{the asymptotic of curvature of direct image bundle associated with higher powers of a relative ample line bundle}
\makeatletter
\makeatother
\author{Xueyuan Wan}
\author{Genkai Zhang}

\address{Xueyuan Wan: Mathematical Sciences, Chalmers University of Technology, 41296 Gothenburg, Sweden}
\email{xwan@chalmers.se}

\address{Genkai Zhang: Mathematical Sciences, Chalmers University of Technology, 41296 Gothenburg, Sweden}
\email{genkai@chalmers.se}

\begin{abstract}
Let $\pi:\mc{X}\to M$ be a holomorphic fibration
 with compact fibers and $L$ a relative ample line bundle over $\mc{X}$. 
We obtain the asymptotic of the curvature of $L^2$-metric and Qullien metric on the direct image bundle $\pi_*(L^k+K_{\mc{X}/M})$ up to the lower order terms than $k^{n-1}$ for large $k$. As an application we prove that
 the analytic torsion $\tau_k(\b{\p})$ satisfies 
  $\p\b{\p}\log(\tau_k(\b{\p}))^2=o(k^{n-1})$, where $n$ is the dimension of fibers.  
 \end{abstract}
\maketitle

\section{Introduction}

Let $\pi:\mc{X}\to M$ be a holomorphic fibration with compact fibers
and  $L$ a relative ample line bundle over $\mc{X}$, i.e. there is a smooth metric (weight) $\phi$ on $L$ such that the first Chern form $\frac{\sqrt{-1}}{2\pi}\p\b{\p}\phi$ is positive $(1,1)$-form along each fiber. One may consider the following direct image bundle $$E^k=\pi_*(L^k+K_{\mc{X}/M}).$$ Here
 $K_{\mc{X}/M}=
K_{\mc{X}}-\pi^*K_M$
 denotes the relative canonical line bundle.
 The bundle
$E^k$ is 
equipped with the canonical $L^2$-metric 
\begin{align}
\|u\|^2:=\int_{\mc{X}_y}|u|^2e^{-\phi}, \quad
u\in E_y^k, \quad y\in M; 
\end{align}
see \cite{Bern2, Bern3, Bern4}. 
Here $|u|^2e^{-\phi}$ is defined as follows: $u$ can be written locally as $u=u' dv\wedge e$, where $e$ is a local holomorphic frame for $L|_{\mc X_y}$, $\mc{X}_y=\pi^{-1}(y)$, and 
$$|u|^2e^{-\phi}:=(\sqrt{-1})^{n^2}|u'|^2 |e|^2dv\wedge d\b{v}=(\sqrt{-1})^{n^2}|u'|^2 e^{-\phi}dv\wedge d\b{v},$$
where $dv=dv^1\wedge \cdots\wedge dv^n$.

In \cite{Bern2, Bern4} Berndtsson  
computed the curvature  $\Theta^{E_k}$ of the $L^2$-metric and consequently proved
Nagano positivity.
More precisely,
\begin{align}\label{0.4}
\langle\sqrt{-1}\Theta^{E_k}u,u\rangle=\int_{\mc{X}/M}kc(\phi)|u|^2e^{-k\phi}+k\langle (\Delta'+k)^{-1} i_{\mu_{\alpha}}u, i_{\mu_{\beta}}u\rangle\sqrt{-1}dz^{\alpha}\wedge d\b{z}^{\beta},	
\end{align}
where the definitions of $c(\phi)$, $\mu_{\alpha}$ and $\Delta'$ 
are given in Theorem \ref{BF1}, (see e.g. \cite[Theorem 1.2]{Bern4}). 

In particular, if $\pi:\mc{X}\to M$ is a trivial fibration, Berndtsson \cite[Theorem 4.1, 4.2]{Bern3} obtained an asymptotic of $tr\Theta^{E_k}/d_k$ up to $o(1)$, $d_k=\text{rank} E_k$.  
In view of the relation of the $L^2$-curvature with
analytic torsion and  Quillen metric
it is a natural and interesting problem to find 
the lower order terms in the asymptotic of $tr\Theta^{E_k}$ for a general fibration.  We solve the problem
up to reminder term of order $o(k^{n-1})$.
\begin{thm}\label{main theorem 1}
For any vector $\zeta\in T_yM$, we have 
	\begin{align}\label{0.7}
\begin{split}
	&\quad -\sqrt{-1}c_1(E^k,\|\bullet\|_k)(\zeta,\b{\zeta})\\
	&=\frac{k^{n+1}}{(2\pi)^{n+1}}\int_{\mc{X}_y}(-\sqrt{-1})c(\phi)(\zeta,\b{\zeta})\frac{\omega^{n}}{n!}+\frac{k^n}{(2\pi)^{n+1}}\int_{\mc{X}_y}\left(\frac{1}{2}|\mu|^2-\frac{\rho}{2}(-\sqrt{-1})c(\phi)(\zeta,\b{\zeta})\right)\frac{\omega^n}{n!}\\
	&+\frac{k^{n-1}}{(2\pi)^{n+1}}\int_{\mc{X}_y}\left((-\sqrt{-1})c(\phi)(\zeta,\b{\zeta})\left(-\frac{1}{6}\Delta\rho+\frac{1}{24}(|R|^2-4|Ric|^2+3\rho^2)\right)-\frac{\rho}{4}|\mu|^2\right)\frac{\omega^n}{n!}\\
	&+\frac{k^{n-1}}{(2\pi)^{n+1}}\left(\frac{1}{12}\|\mu\|^2_{Ric}+\frac{1}{12}\|\n'\mu\|^2-\frac{1}{4}\|\b{\p}^*\mu\|^2\right)+o(k^{n-1}).
\end{split}	
\end{align}
\end{thm}
We refer to
the subsection \ref{Asy1} for the definitions of the notation in (\ref{0.7}).

We note that the first two terms in the expansion above were proved by Ma-Zhang
\cite{Ma}. The leading terms of the first summand in
(\ref{0.4}) is studied by Sun \cite{Sun}
and the second by Berndtsson \cite{Bern3} (in the different
setup of trivial fibration with variation of K\"a{}hler metrics).

We shall then compare  our expansion
for the $L^2$-curvature with the Quillen curvature.
So let $D_y=\b{\p}_y+\b{\p}^*_y$  be the Dirac operator 
acting on 
$A^{0,*}(\mc{X}_y, L^{k}+K_{\mc{X}/M})$ of $(0, \ast)$-forms, where $\mc{X}_y$ is endowed
with the restricted Hermitian metric 
$(\sqrt{-1}\p\b{\p}\phi)|_{\mc{X}_y}$.
 For any $0<b<c$, denote by $D_i^{(b,c)}$ the restriction of $D$ on the sum of eigenspaces of $A^{0,i}(\mc{X}_y,  L^{k}+K_{\mc{X}/M})$ for eigenvalues in $(b,c)$,  $K_{\mc{X}/M}$ being equipped a natural Hermitian metic 
$$(\det\phi)^{-1}:=(\det(\p\b{\p}\phi|_{\mc{X}_y}))^{-1}.$$ Then the (Ray-Singer) analytic torsion is defined by
\begin{align*}
\tau_k(\b{\p})&=\tau_k(\b{\p}^{(b,+\infty)})\\
&=\left((\det (D_1^{(b,+\infty)})^2)(\det (D_2^{(b,+\infty)})^2)^{-2}(\det (D_3^{(b,+\infty)})^2)^3\cdots\right)^{1/2},
\end{align*}
and is a positive smooth function on the base manifold $M$.  Here $b$ is a constant less than all positive eigenvalues of $D$ (see Definition \ref{torsiondef}).

The Quillen metric $\|\bullet\|_Q$ on the determinant line $\lambda$ (see Definition \ref{line bundle}) is patched by the $L^2$-metric $|\bullet|^b$ on $\lambda^b$ (see (\ref{lambdab})) and the analytic torsion $\tau_k(\b{\p})$, i.e.
\begin{align}\label{0.9}
	\|\bullet\|_Q=|\bullet|^b\tau_k(\b{\p}),
\end{align}
where $b>0$ is a sufficiently small constant.

In their papers \cite{Bismut1,Bismut2, Bismut3}, 
J.-M. Bismut, H. Gillet and C. Soul\'{e} computed the curvature of Quillen metric for a locally K\"ahler family and obtained the differential form version of Grothendieck-Riemann-Roch Theorem. More precisely, 
 they  proved that as holomorphic bundles,
$$\lambda_y\cong \bigotimes_{i\geq 0}\det H^i(\mc{X}_y, L^{k}+K_{\mc{X}/M})^{(-1)^{i+1}}$$ and the curvature is
\begin{align}\label{0.5}
c_1(\lambda, \|\bullet\|_Q)=
-\left\{\int_{\mc{X}/M} Td\left(\frac{-R^{T_{\mc{X}/M}}}{2\pi i}\right)Tr\left[\exp\left(\frac{-R^{ L^{k}+K_{\mc{X}/M}}}{2\pi i}\right)\right]\right\}^{(1,1)}.
\end{align}

Since $L$ is a relative ample line bundle over $\mc{X}$,
 by Kodaira vanishing theorem, $$H^i(\mc{X}_y, K_{\mc{X}/M}+L^k)=0$$ for all $i\geq 1$. Therefore,
 \begin{align}\label{0.8}
 \lambda\cong (\det E^k)^{-1}.
 \end{align}
We expand also the Quillen curvature $c_1(\lambda, \|\bullet\|_Q)$ and compare
it with (\ref{0.7}). We prove {
 \begin{thm}\label{main theorem 2}
Up to  terms of order $o(k^{n-1})$ 
the Quillen curvature  $-c_1(\lambda,\|\bullet\|_Q)$ has the same expansion
 (\ref{0.7})  as for the $L^2$-curvature.
 \end{thm}}

As application we shall find the asymptotics of the variation of the analytic torsion.
  From (\ref{metric consist}) we have
\begin{align}\label{0.10}
\det \|\bullet\|_k^2=((|\bullet|^b)^2)^*\end{align}
for $b>0$ a sufficiently small constant, where $\det \|\bullet\|_k$ denotes  the natural induced $L^2$-metric on line bundle $\det E^k$ and $((|\bullet|^b)^2)^*$ denotes the dual metric of $(|\bullet|^b)^2$. Using (\ref{0.9}) and (\ref{0.10}) we have furthermore
\begin{align}
\frac{\sqrt{-1}}{2\pi}\p\b{\p}\log (\tau_k(\b{\p}))^2
=-c_1(\lambda,\|\bullet\|_Q)-c_1(E^k, \|\bullet\|_k).
\end{align}

As an immediate consequence
of  Theorem \ref{main theorem 1} and Theorem \ref{main theorem 2}
we have

\begin{cor}\label{cor1}
 As $k\to \infty$, we have
	 \begin{align}\label{second variation}
 \p\b{\p}\log(\tau_k(\b{\p}))^2=o(k^{n-1}).
 \end{align} 
\end{cor}

Here the asymptotic (\ref{second variation}) is understood
as  $(\p\b{\p}\log(\tau_k(\b{\p}))^2)(\zeta,\b{\zeta})=o(k^{n-1})$ for any vector $\zeta\in TM$. 
\begin{rem} The asymptotic of analytic torsion has been studied by \cite{Bismut0}, \cite{Berman}. It is proved in \cite[Theorem 8]{Bismut0} 
the coefficients of $k^n$, $k^n\log k$ are topological 
invariants. After a preliminary version of this paper was finished,
we were informed by Xiaonan Ma of the paper
 \cite[Theorem 1.1, 1.2]{Finski} by Finski
where the coefficients of $k^{n-1}, k^{n-1}\log k$ 
in the analytic torsion $\tau_k$ have also been computed, which implies then (\ref{second variation}). 
 However, our method here is completely different 
from the methods in \cite{Bismut0, Finski}. 
\end{rem}

\begin{rem}
For the case of $M$ is the Teichm\"uller space of compact Riemann surfaces of genus $g\geq 2$ and $L$ is the relative canonical line bundle, 
Corollary \ref{cor1} was proved in \cite{Zhang}; in this
case the analytic torsion 
 $\tau_k$ actually decays exponentially in $k$.
\end{rem}

We proceed to explain briefly our 
method for the expansion in  Theorem \ref{main theorem 1}.
By the formula (\ref{0.4}),
the first Chern form  $c_1(E^k, \|\bullet\|_k)$ 
is the trace
of an integral operator, and in the
paper \cite{Ma} X. Ma and W. Zhang 
found the expansion of the diagonal of the 
  kernel  of the operators, proving a local index formula. To find the  third order term, i.e. the coefficient of $k^{n-1}$, 
seems a difficult task and requires much more effort. 
The trace
of first summand in (\ref{0.4}) is relatively easy to handle, 
however the second summand involves Toeplitz operators
with symbols being differential operators. A major ingredient of our method
 is the following expansion (see Lemma \ref{expanding}),
\begin{align}\label{0.3}
\begin{split}
(\Delta'+k)^{-1}&=\frac{1}{2k}+\frac{1}{6k^2}(k-\Delta'-R^*)+\frac{1}{4k^2}R^*+\frac{1}{4k^2}(\Delta'+k)^{-1}(k-\Delta')R^*\\
&\quad+\frac{1}{18k^3}(2k-\Delta')(k-\Delta'-R^*)+\frac{1}{36k^4}(2k-\Delta')^2(k-\Delta'-R^*)\\
&\quad+\frac{1}{36k^4}(\Delta'+k)^{-1}(k-\Delta')(2k-\Delta')^2(k-\Delta'-R^*),
\end{split}
\end{align}
where $R^*=R^{\b{t}j}_{i\b{l}}i_{\frac{\p}{\p v^j}}dv^i\wedge d\b{v}^l\wedge i_{\frac{\p}{\p\b{v}^t}}$, $ R^{\b{t}j}_{i\b{l}}=-\p_{i}\p_{\b{l}}\phi^{\b{t}j}+\phi^{\b{t}s}_{\b{l}}\phi^{\b{k}j}_i\phi_{s\b{k}}$ is the Chern curvature component of $(T^*_X,(\phi^{i\b{j}}))$.
 The corresponding contribution of each term above
to the $L^2$-curvature  $c_1(E^k, \|\bullet\|_k)$ will be effectively treated by using
further the asymptotic expansion of Bergman Kernel for bundles \cite[Theorem 4.2]{Wang}.

This article is organized as follows. In Section \ref{sec1}, we fix notation
and recall some basic facts  on  Berndtsson's  curvature formula of $L^2$-metric, the asymptotic expansion of Bergman kernel for bundles, analytic torsion and Quillen metric. In Section \ref{sec2}, we  find
 the expansion of $c_1(E^k, \|\bullet\|_k)$ and prove Theorem \ref{main theorem 1}. We also give the expansion of $-c_1(\lambda, \|\bullet\|_Q)$ and prove Theorem \ref{main theorem 2}. By comparing with their expansions, we prove Corollary \ref{cor1}.

\vspace{3mm} 
{\bf Acknowledgements.}
We would like to thank Professor Bo Berndtsson for
 patiently answering us several
questions and for inspiring discussions, and Professor Xiaonan Ma for 
drawing our attention to  the  preprint \cite{Finski}.
We are also grateful to Siarhei Finski
 who explained us carefully his results in 
a written communication.

\section{Preliminaries}\label{sec1}

We shall fix notation and recall some necessary background material.

\subsection{Berndtsson's curvature formula of $L^2$-metric}

We refer \cite{Bern2, Bern3, Bern4} and references therein.

 Let $\pi:\mc{X}\to M$ be a holomorphic fibration with compact fibres and $L$ a relative ample line bundle over $\mc{X}$. We denote by
$(z;v)=(z^1,\cdots, z^m; v^1,\cdots, v^n)$ a local admissible holomorphic coordinate system of $\mc{X}$ with $\pi(z;v)=z$, where $m=\dim_{\mb{C}}M$, $n=\dim_{\mb{C}}\mc{X}-\dim_{\mb{C}}M$.

For any smooth function $\phi$ on $\mc{X}$, we denote
$$\phi_{\alpha}:=\frac{\p \phi}{\p z^{\alpha}},\quad \phi_{\b{\beta}}:=\frac{\p \phi}{\p \b{z}^{\beta}},
\quad \phi_{i}:=\frac{\p \phi}{\p v^i},\quad \phi_{\b{j}}:=\frac{\p \phi}{\p \b{v}^j},$$
where $1\leq i,j\leq n, 1\leq \alpha,\beta\leq m$.

Let $F^+(L)$ be the space of smooth metrics $\phi$ on $L$ with
\begin{align*}
(\sqrt{-1}\p\b{\p}\phi)|_{\mc{X}_y}>0	
\end{align*}
for any point $y\in M$. For any $\phi\in F^+(L)$, set
\begin{align}\label{horizontal}
  \frac{\delta}{\delta z^{\alpha}}:=\frac{\p}{\p z^{\alpha}}-\phi_{\alpha\b{j}}\phi^{\b{j}k}\frac{\p}{\p v^k}.
\end{align}
By a routine computation, one can show that $\{\frac{\delta}{\delta z^{\alpha}}\}_{1\leq \alpha\leq m}$ spans a well-defined horizontal subbundle of $T\mc{X}$. 

 Let $\{dz^{\alpha};\delta v^k\}$
denote the dual frame of $\left\{\frac{\delta}{\delta z^{\alpha}}; \frac{\p}{\p v^i}\right\}$. Then
$$\delta v^k=dv^k+\phi^{k\b{l}}\phi_{\b{l}\alpha}dz^{\alpha}.$$
Moreover, the differential operators
\begin{align}\label{HV}
\p^V=\frac{\p}{\p v^i}\otimes \delta v^i,\quad \p^H=\frac{\delta}{\delta z^{\alpha}}\otimes dz^{\alpha}.
\end{align}
are well-defined.

For any $\phi\in F^+(L)$, the geodesic curvature $c(\phi)$ is defined by
\begin{align*}
  c(\phi)=\left(\phi_{\alpha\b{\beta}}-\phi_{\alpha\b{j}}\phi^{i\b{j}}\phi_{i\b{\beta}}\right)\sqrt{-1} dz^{\alpha}\wedge d\b{z}^{\beta},
\end{align*}
which is a horizontal real $(1,1)$-form on $\mc X$. 
The following lemma
confirms that the geodesic curvature $c(\phi)$ of $\phi$ is indeed well-defined.
\begin{lemma}[\cite{Feng}]\label{lemma0} The following decomposition holds,
  \begin{align*}
    \sqrt{-1}\p\b{\p}\phi=c(\phi)+\sqrt{-1}\phi_{i\b{j}}\delta v^i\wedge \delta \b{v}^j.
  \end{align*}
\end{lemma}
\begin{proof}
This is a direct computation,
  \begin{align*}
    &c(\phi)+\sqrt{-1}\phi_{i\b{j}}\delta v^i\wedge \delta \b{v}^j=\sqrt{-1}(\phi_{\alpha\b{\beta}}-\phi_{\alpha\b{l}}\phi^{k\b{l}}\phi_{k\b{\beta}})dz^{\alpha}\wedge d\b{z}^{\beta}\\
    &\quad+\sqrt{-1}\phi_{i\b{j}}(dv^i+\phi^{i\b{l}}\phi_{\b{l}\alpha}dz^{\alpha})\wedge (d\b{v}^j+\phi^{\b{j}k}\phi_{k\b{\beta}}d\b{z}^{\beta})\\
    &=\sqrt{-1}(\phi_{\alpha\b{\beta}}dz^{\alpha}\wedge d\b{z}^{\beta}+\phi_{\alpha\b{j}}dz^{\alpha}\wedge d\b{v}^j+\phi_{i\b{\beta}}dv^{i}\wedge d\b{z}^{\beta}+\phi_{i\b{j}}dv^i\wedge d\b{v}^j)\\
    &=\sqrt{-1}\p\b{\p}\phi.
  \end{align*}
\end{proof}

Following Berndtsson (cf. \cite{Bern2, Bern3, Bern4}) we consider the direct image bundle $E:=\pi_*(K_{\mc{X}/M}+L)$, and define the following $L^2$-metric on
the direct image bundle $E:=\pi_*(K_{\mc{X}/M}+L)$: 
for $y\in M$, $\mc{X}_y=\pi^{-1}(y)$,  and  $u\in E_{y}\equiv H^0(\mc{X}_y, (L+K_{\mc{X}/M})_y)$, 
\begin{align}\label{L2 metric}
\|u\|^2:=\int_{\mc{X}_y}|u|^2e^{-\phi}. 	
\end{align}
Note that $u$ can be written locally as $u=u' dv\wedge e$, where $e$ is a local holomorphic frame for $L|_{\mc X}$, and so locally
$$|u|^2e^{-\phi}=(\sqrt{-1})^{n^2}|u'|^2 |e|^2dv\wedge d\b{v}=(\sqrt{-1})^{n^2}|u'|^2 e^{-\phi}dv\wedge d\b{v},$$
where $dv=dv^1\wedge \cdots\wedge dv^n$ is the fiber volume.

By the definition of $\b{\p}^V$, we have
\begin{align*}
\mu_{\alpha}=\b{\p}^V(\frac{\delta}{\delta z^{\alpha}})=-\frac{\p}{\p \b{v}^l}\left(\phi_{\alpha\b{j}}\phi^{\b{j}i}\right)d\b{v}^l\otimes \frac{\p}{\p v^i}, 	
\end{align*}
which is in the Kodira-Spencer class $\rho(\frac{\p}{\p z^{\alpha}}|_y)\in H^1(\mc{X}_y, T_{\mc{X}_y})$. 

 The following theorem was proved by Berndtsson in \cite[Theorem 1.2]{Bern4}, its proof also can be found in \cite[Theorem 3.1]{Feng}.

\begin{thm}[\cite{Bern4}]\label{BF1}
\label{thm4} For any $y\in M$ 
the curvature 
 $\langle \Theta^{E}(u, u)\rangle$, $ u\in E_{y}$, 
  of the Chern connection on $E$ with the $L^2$ metric 
is given by
\begin{align}\label{cur}
\langle \sqrt{-1}\Theta^{E}u,u\rangle=\int_{\mc{X}_y}c(\phi)|u|^2e^{-\phi}+\langle(1+\Delta')^{-1}i_{\mu_{\alpha}}u,i_{\mu_{\beta}}u\rangle\sqrt{-1}dz^{\alpha}\wedge d\b{z}^{\beta}.
\end{align}
Here $\Delta'=\n'\n'^*+\n'^*\n$ is the Laplacian on $L|_{\mc{X}_y}$-valued forms on $\mc{X}_y$ defined by the $(1,0)$-part of the Chern connection on $L|_{\mc{X}_y}$.
\end{thm}

We replace now the Hermitian line bundle $(L,e^{-\phi})$ by $(L^k,e^{-k\phi})$, 
and consider the corresponding direct image bundle $E^k:=\pi_*(L^k+K_{\mc{X}/M})$. Let $\n'^*_k$ (resp. $\n'^*$) be the adjoint operator of  $\n'$ with respect to $(L^k,e^{-k\phi})$ and $(X, k\omega=k\sqrt{-1}\p\b{\p}\phi)$ (resp. $(X, \omega=\sqrt{-1}\p\b{\p}\phi)$). We have
\begin{align}
\n'^*=\sqrt{-1}[\Lambda_{k\omega},\n']=\frac{1}{k}\sqrt{-1}[\Lambda_{\omega},\n']=\frac{1}{k}\n'^*.	
\end{align}
It implies that 
\begin{align}\label{delt1}
\Delta'_k=\n'^*_k\n'+\n'\n'^*_k=\frac{1}{k}\Delta'.	
\end{align}

 From Theorem \ref{BF1} and (\ref{delt1}), the curvature of $L^2$-metric (see (\ref{L2 metric})) on $E^k$ is given by 
\begin{align}\label{L2 curvature}
\begin{split}
\langle \sqrt{-1}\Theta^{E^k}u,u\rangle &=\int_{\mc{X}_y}c(k\phi)|u|^2e^{-k\phi}+\langle (1+\Delta'_k)^{-1}i_{\mu_{\alpha}}u,i_{\mu_{\beta}}u\rangle_{k\omega}\sqrt{-1}dz^{\alpha}\wedge d\b{z}^{\beta}\\
&=\int_{\mc{X}_y}kc(\phi)|u|^2e^{-k\phi}+k\langle (k+\Delta')^{-1}i_{\mu_{\alpha}}u,i_{\mu_{\beta}}u\rangle\sqrt{-1}dz^{\alpha}\wedge d\b{z}^{\beta}
\end{split}
\end{align}
for any element $u$ of $E^k_y$.

\subsection{The Bergman kernels} 
Let $(X,\omega)$ be a compact K\"ahler manifold of $n$-dimension, $(L,e^{-\phi})$ be a Hermitian line bundle over $X$ satisfying 
\begin{align}
	\sqrt{-1}R^L=\sqrt{-1}\p\b{\p}\phi=\omega. 
\end{align}

Let $(E,H)$ be a Hermitian vector bundle over $X$. 
There is a natural $L^2$-metric on the space $H^0(X, L^k+E)$ 
of holomorphic forms of $L^k +E$,  $k\ge 0$. Let $\{u_j\}_{j=1}^{N_k}$ be an orthonormal basis of $H^0(X, L^k+E)$, where $d_k=\dim H^0(X, L^k+E)$. The $k$-th Bergman kernel $B_k(H)\in \text{End}(E)$ is defined by 
\begin{align}
B_k(H)=\sum_{j=1}^{d_k}u_j^*\otimes u_j.
\end{align}

Recall the following Tian-Yau-Zelditch expansion of Bergman kernel for bundles.
We use the version in  \cite{Catlin} and
refer \cite{Berman1, Dai, 
Lu, 
Tian, Wang, Xu,
 Zelditch} for different variations and proofs.

\begin{thm}\label{cat}
For a fixed metric $H$, there is an asymptotic expansion as $k\to \infty$,
$$B_k(H)=\frac{1}{(2\pi)^n}(A_0 k^n+A_1k^{n-1}+\cdots),$$
where $A_i\in\text{End}(E)$ are determined by the geometry of $\omega$ and $H$. The expansion is in the sense that for any integer $l,R\geq 0$,
$$\|(2\pi)^nB_k(H)-\sum_{j<R}A_j k^{n-j}\|_{C^l}\leq C_{l,R,H}k^{n-R},$$
where the norm is computed in the space $C^l(X, \text{End}(E))$
of $\text{End}(E)$-valued sections and 
 $C_{l,R,H}$ depends on $l,R,\omega$ and $H$. 	
\end{thm}

The first three coefficients $A_0$, $A_1$ and $A_2$ have
been computed in \cite[Theorem 4.2]{Wang}. 
\begin{thm}[\cite{Wang}]
	\begin{itemize}
	\item[(0)] $A_0=Id$,
	\item[(1)]	$A_1=\sqrt{-1}\Lambda F_H+\frac{1}{2}\rho Id$,
	\item[(2)]$A_2=\frac{1}{3}\Delta\rho+\frac{1}{24}(|R|^2-4|Ric|^2+3\rho^2)
	+\frac{1}{2}(\Delta'' Ric^{E}+\rho Ric^{E}+Ric^{E}Ric^{E}-R^{E}R^{E}-\langle R^{E}, Ric\rangle$.
	\end{itemize}
	Here $R$, $Ric$ and $\rho$ represent the curvature tensor, the Ricci curvature and the scalar curvature of $\omega$, and $\Delta=\phi^{\b{j}i}\frac{\p^2}{\p v^i\p\b{v}^j}$, $Ric^E=\sqrt{-1}\Lambda F_H$ and $F_H$ represents the curvature of $(E,H)$, $\Delta''=\sqrt{-1}\Lambda\p\b{\p}$. 
\end{thm}

When  $E=K_X$ is the canonical bundle there is a natural metric $(\det(\phi_{i\b{j}}))^{-1}$ on $K_X$ induced from $(L,e^{-\phi})$. In this case, $Ric^E=-\rho$, $R^E=-Ric$, so 
\begin{align}\label{bergman1}
A_0=1,\quad A_1=-\frac{\rho}{2},\quad A_2=-\frac{1}{6}\Delta''\rho+\frac{1}{24}(|R|^2-4|Ric|^2+3\rho^2).
\end{align}
The Bergman kernel is
\begin{align}\label{bergman2}
B_k(H)=\sum_{j=1}^{d_k}u_j^*\otimes u_j=\sum_{j=1}^{d_k}|u_j|_{L^2}^2.
\end{align}

From (\ref{bergman1}), (\ref{bergman2}) and Theorem \ref{cat}, the asymptotic expansion of Bergman kernel for the bundle $L^k+K_X$ is 
\begin{align}\label{Bergman expansion}
\begin{split}
	\sum_{j=1}^{d_k}|u_j|_{L^2}^2
	=\frac{1}{(2\pi)^n}\left(k^n-\frac{\rho}{2}k^{n-1}+(-\frac{1}{6}\Delta''\rho+\frac{1}{24}(|R|^2-4|Ric|^2+3\rho^2))k^{n-2}+\cdots\right).
	\end{split}
\end{align}
 
\subsection{Analytic torsion and Quillen metric}\label{sub1}

The definitions and results in this subsection
can be found in  \cite{BGV, Bismut1, Bismut2, Bismut3,  Ma1, Ray}.

 Let $\pi:\mc{X}\to M$ be a  proper holomorphic mapping between complex manifolds $\mc{X}$ and $M$, 
$(F,h_F)$  a holomorphic Hermitian vector bundle on
 $\mc{X}$, $\n^{F}$ the corresponding Chern connection, and
$R^{F}=(\n^{F})^2$ its curvature. 
For any $y\in M$, let $\mc{X}_y=\pi^{-1}(y)$ be the fiber over $y$ with 
 K\"ahler metric $g^{\mc{X}_y}$ depending smoothly on $y$. The fibers are assumed
to be compact.

 For any $0\leq p\leq n:=\dim_{\mb{C}}\mc{X}_y$ we put
\begin{align*}
E^p_y:=A^{0,p}(\mc{X}_y,F),\quad E_y=\bigoplus_{p\geq 0}E^p_y.	
\end{align*}
The operator $D_y=\b{\p}_y+\b{\p}^*_y$ acts on the fiber $E_y$.

For every $y\in M$, the spectrum of $D^2_y$ is discrete. For $b>0$, let $K^{b,p}_y$ be the sum of the eigenspaces of the operator  $D^2_y$ acting on $E^p_y$ for eigenvalues $<b$.
Let $U^b$ be the open set:
$$U^b=\{y\in M; b\not\in \text{Spec}D^2_y\}.$$

On each open set $U^b$, $K^{b,p}$ is a smooth finite dimensional vector bundle. Set
\begin{align*}
K^{b,+}=\bigoplus_{p\,\text{even}}K^{b,p},\quad K^{b,+}=\bigoplus_{p\,\text{odd}}K^{b,p},\quad K^b=K^{b,+}\oplus K^{b,-}.	
\end{align*}
Define the following line bundle $\lambda^b$ on $U^b$,
\begin{align}\label{lambdab}\lambda^b=(\det K^{b,0})^{-1}\otimes (\det K^{b,1})\otimes (\det K^{b,2})^{-1}\otimes\cdots.\end{align}
For $0<b<c$, if $y\in U^b\cap U^c$, let $K^{(b,c),p}_y$ be the sum of eigenspaces of $D^2_y$ in $E^p_y$ for eigenvalues $\mu$ such that $b<\mu<c$. Set
\begin{align*}
K^{(b,c),+}_y=\bigoplus_{p\,\text{even}}K^{(b,c),p},\quad K^{(b,c),-}_y=\bigoplus_{p\,\text{odd}}K^{(b,c),p},\quad 	K^{(b,c)}_y=K^{(b,c),+}_y\oplus K^{(b,c),-}_y.
\end{align*}
Define $\lambda^{(b,c)}$ accordingly  as before. Let $\b{\p}^{(b,c)}$ and $D^{(b,c)}$ be the restriction of $\b{\p}$ and $D$ to $K^{(b,c)}$.  $D^{(b,c)}_{\pm}$ is the restriction of $D$ to $K^{(b,c),\pm}$.

Since the chain complex
\begin{align}\label{sequence}
0\to K^{(b,c),0}\xrightarrow{\b{\p}^{(b,c)}}K^{(b,c),1}	\xrightarrow{\b{\p}^{(b,c)}}\cdots \xrightarrow{\b{\p}^{(b,c)}}K^{(b,c),n}\to 0
\end{align}
is acyclic,  $\lambda^{(b,c)}$ has a canonical non-zero section $T(\b{\p}^{(b,c)})$ which is smooth on $U^b\cap U^c$ (see \cite[Definition 1.1]{Bismut1}). For $0<b<c$, over $U^b\cap U^c$, we have the $C^{\infty}$ identifications
\begin{align*}
\lambda^c=\lambda^b\otimes \lambda^{(b,c)}. 	
\end{align*}
We identify $\lambda^b$ and $\lambda^c$ over $U^b\cap U^c$ by the $C^{\infty}$ map
\begin{align}\label{tran}
s\in \lambda^b\mapsto s\otimes T(\b{\p}^{(b,c)})\in \lambda^c.
\end{align}

 \begin{defn}(\cite[Def. 1.1]{Bismut3})\label{line bundle}
 {\it The $C^{\infty}$ line bundle  $\lambda$ over $M$ 
is $\{(U^b, \lambda^b)\}$  with the transition functions (\ref{tran}) on $U^b\cap U^c$.}
 \end{defn}

The analytic torsion of the chain complex (\ref{sequence}) was  introduced by
Ray and Singer \cite{Ray}.
\begin{defn}\label{torsiondef}
	{\it The analytic torsion $\tau(\b{\p}^{(b,c)})$ associated to the acyclic chain complex (\ref{sequence}) is  defined as the positive real number
\begin{align*}
\tau(\b{\p}^{(b,c)})=\left((\det (D_1^{(b,c)})^2)(\det (D_2^{(b,c)})^2)^{-2}(\det (D_3^{(b,c)})^2)^3\cdots\right)^{1/2},
\end{align*}
where $D^{(b,c)}_p$ is the restriction of $D$ to $K^{(b,c), p}$, $1\leq p\leq n$.  If $b$ is a small constant less than all positive eigenvalues of $D^2_y$, then we denote
\begin{align*}
\tau(\b{\p}):=\tau(\b{\p}^{(b,+\infty)}).	
\end{align*}}
\end{defn}

 Since $K^b$ and $K^{(b,c)}$ are orthogonal subspaces of $K^c$, by \cite[Proposition 1.5]{Bismut1}, we find that if $s\in \lambda^b$,
\begin{align*}
|s\otimes T(\b{\p}^{(b,c)})|^c=|s|^b\tau(\b{\p}^{(b,c)}),	
\end{align*}
where $|\cdot|^b$ is the induced metric by $(\mc{X}_y, g^{\mc{X}_y})$ and $(F, h_F)$.

Now let $N_V$ be the number operator on $E$ such that $N_V\eta=p\eta$ for $\eta\in E^p$. Set $Q^b=I-P^b$, where $P^b$ is the orthogonal projection operator from $E_y$ on $K^b_y$.

For $y\in U^b$, $\text{Re}(s)>l$, set
\begin{align*}
\theta^b_y(s)=-Tr_s[N_V[D^2]^{-s}Q^b]=\frac{-1}{\Gamma(s)}\int^{+\infty}_0 u^{s-1}Tr_s[N_V\exp(-uD^2)Q^b]du.	
\end{align*}
	Similar if $0<b<c<+\infty$, for $y\in U^b\cap U^c$, set
	\begin{align*}
	\theta^{(b,c)}_y(s)=-Tr_s[N_V[D^2]^{-s}P^{(b,c)}]=\frac{-1}{\Gamma(s)}	\int^{+\infty}_0 u^{s-1}Tr_s[N_V\exp(-uD^2)P^{(b,c)}]du.
	\end{align*}
	
The functions $\theta^b_y$ and $\theta^{(b,c)}_y$ extend into a meromorphic function which is holomorphic at $s=0$. Also on $U^b\cap U^c$,
\begin{align*}
\theta^b=\theta^{(b,c)}+\theta^c.	
\end{align*}
and by \cite[Equation 1.32]{Bismut3},
\begin{align*}
\log (\tau^2(\b{\p}^{(b,c)}))=-\theta^{(b,c)'}(0).	
\end{align*}

For $y\in U^b$, denote
\begin{align*}
\tau_y(\b{\p}^{(b,+\infty)})=\exp\left(-\frac{1}{2}\theta^{b'}_y(0)\right).
\end{align*}
Let $\|\bullet\|^b$ denote the metric on the line bundle $(\lambda^b, U^b)$,
\begin{align}\label{torsion}
\|\bullet\|^b=|\bullet|^b\tau_y(\b{\p}^{(b,+\infty)}).	
\end{align}

Then the definition of Quillen metric $\|\bullet\|_Q$ and  Chern form $c_1(\lambda, \|\bullet\|_Q)$ of the Quillen metric are given by the following theorem.
\begin{thm}[\cite{Bismut1, Bismut2, Bismut3}]\label{Bis}
The metrics $\|\bullet\|^b$ on $(\lambda^b, U^b)$ patch into a smooth Hermitian metric $\|\bullet\|_Q$ on the holomorphic line bundle $\lambda$.	
The Chern form of Hermitian line bundle $(\lambda, \|\bullet\|_Q)$ is
\begin{align}\label{Chern forms}
c_1(\lambda, \|\bullet\|_Q)=-\left\{\int_{\mc{X}/M} Td\left(\frac{-R^{T_{\mc{X}/M}}}{2\pi i}\right)Tr\left[\exp\left(\frac{-R^{F}}{2\pi i}\right)\right]\right\}^{(1,1)}.
\end{align}
\end{thm}

The Knudsen-Mumford determinant is defined by
\begin{align*}
\lambda^{KM}=(\det R\pi_*F)^{-1}.	
\end{align*}
The fiber $\lambda^{KM}_y$ is by definition given by
\begin{align*}
\lambda^{KM}_y=\bigotimes_{i\geq 0}\det H^i(\mc{X}_y,F)^{(-1)^{i+1}}.	
\end{align*}

We assume that $\pi$ is locally K\"ahler, i.e. there is an open covering $\mathscr{U}$ of $M$ such that if $U\in \mathscr{U}$, $\pi^{-1}(U)$ admits a K\"ahler metric.

\begin{thm}[\cite{Bismut1, Bismut2, Bismut3}]\label{Bis11}
Assume that $\pi$ is locally K\"ahler. Then the identification of the fibers $\lambda_y\cong \lambda^{KM}_y$ defines a holomorphic isomorphism of line bundles $\lambda\cong \lambda^{KM}$. The Chern form of the Quillen metric on $\lambda\cong \lambda^{KM}$ is given by
(\ref{Chern forms}).
\end{thm}

\section{The asymptotic of the curvature of direct image bundle}\label{sec2}

We shall give the expansion of $c_1(E^k, \|\bullet\|_k)$ and $-c_1(\lambda, \|\bullet\|_Q)$ up to $o(k^{n-1})$.

Let $\pi: \mc{X}\to M$ be a holomorphic fibration with compact fibers
and $L$ a relative ample line bundle over $\mc{X}$
as in the Section 2.1. Denote
\begin{align}\label{omega}
\omega=\sqrt{-1}\p\b{\p}\phi.	
\end{align}

\subsection{The asymptotic of the curvature of $L^2$-metric}\label{Asy1}
The curvature of direct image bundle $E^k=\pi_*(L^k+K_{\mc{X}/M})$ is,
by (\ref{L2 curvature}), 
\begin{align*}
\langle\sqrt{-1}\Theta^{E^k}u,u\rangle=\int_{\mc{X}_y}kc(\phi)|u|^2e^{-k\phi}+k\langle (k+\Delta')^{-1}i_{\mu_{\alpha}}u,i_{\mu_{\beta}}u\rangle\sqrt{-1}dz^{\alpha}\wedge d\b{z}^{\beta}
\end{align*}
for any element $u\in E^k_y$. For any vector $\zeta=\zeta^{\alpha}\frac{\p}{\p z^{\alpha}}$ of $TM$,
\begin{align}\label{2.1}
\begin{split}
	\langle\Theta^{E^k}u,u\rangle(\zeta,\b{\zeta})&=-\sqrt{-1}\left(\int_{\mc{X}_y}kc(\phi)|u|^2e^{-k\phi}\right)(\zeta,\b{\zeta})\\
	&\quad+k\langle (k+\Delta')^{-1}i_{\mu}u,i_{\mu}u\rangle,
	\end{split}
\end{align}
where 
\begin{align}\label{definition of mu}
\mu=\mu_{\b{j}}^id\b{v}^j\otimes \frac{\p}{\p v^i}, \quad \mu^i_{\b{j}}=-\p_{\b{j}}(\phi_{\alpha\b{l}}\phi^{\b{l}i})\zeta^{\alpha}.	
\end{align}

 The following technical expansion of $(\Delta' +k)^{-1}$ will be
critical to find the asymptotics of the $L^2$-curvature; 
the main point of this expansion is that the leading
term of the contribution to the $L^2$-curvature
of each term below is effectively found. 
(Composed with the operators $i_\mu$
and $i_\mu^\ast
$ it gives an expansion of the Toeplitz operator  with symbol
$i_\mu^\ast(\Delta'+k)^{-1}i_\mu$  on the  cohomology 
space $H^{0}(\mc{X}_y,L^k+K_{\mc{X}_y})$.)

\begin{lemma}\label{expanding}
The resolvent operator $(\Delta'+k)^{-1}$ has the following $7$-term-expansion, 
\begin{equation}
\label{2.20}
(\Delta'+k)^{-1}=I + II + \cdots + VI +VII,
\end{equation}
where
\begin{equation*}
I=\frac{1}{2k}, \,\,  II= \frac{1}{6k^2}(k-\Delta'-R^*), \,\,  III=\frac{1}{4k^2}R^*, \,\, IV=
\frac{1}{4k^2}(\Delta'+k)^{-1}(k-\Delta')R^*,
\end{equation*}
\begin{equation*}
V=\frac{1}{18k^3}(2k-\Delta')(k-\Delta'-R^*), 
\, VI= \frac{1}{36k^4}(2k-\Delta')^2(k-\Delta'-R^*),
\end{equation*}
\begin{equation*}
VII=\frac{1}{36k^4}(\Delta'+k)^{-1}(k-\Delta')(2k-\Delta')^2(k-\Delta'-R^*).
\end{equation*}
\end{lemma}
\begin{proof}
	The RHS of (\ref{2.20}), by elementary computations,  is
	\begin{align*}
&\frac{1}{2k}+\frac{1}{6k^2}(k-\Delta'-R^*)+\frac{1}{4k^2}R^*+\frac{1}{4k^2}(\Delta'+k)^{-1}(k-\Delta')R^*\\
&\quad+\frac{1}{18k^3}(2k-\Delta')(k-\Delta'-R^*)+\frac{1}{36k^4}(2k-\Delta')^2(k-\Delta'-R^*)\\
&\quad+\frac{1}{36k^4}(\Delta'+k)^{-1}(k-\Delta')(2k-\Delta')^2(k-\Delta'-R^*)\\
&=\frac{1}{2k}+\frac{1}{6k^2}(k-\Delta'-R^*)+\frac{1}{4k^2}\left((\Delta'+k)^{-1}(k-\Delta')+Id\right)R^*\\
&\quad+\frac{1}{18k^3}(2k-\Delta')(k-\Delta'-R^*)\\
&\quad+\frac{1}{36k^4}\left((\Delta'+k)^{-1}(k-\Delta')+Id\right)(2k-\Delta')^2(k-\Delta'-R^*)\\
&=\frac{1}{2k}+\frac{1}{6k^2}(k-\Delta'-R^*)+\frac{1}{2k}(\Delta'+k)^{-1}R^*\\
&\quad+\frac{1}{18k^3}(2k-\Delta')(k-\Delta'-R^*)\\
&\quad+\frac{1}{18k^3}(\Delta'+k)^{-1}(2k-\Delta')^2(k-\Delta'-R^*)\\
&=\frac{1}{2k}+\frac{1}{6k^2}(k-\Delta'-R^*)+\frac{1}{2k}(\Delta'+k)^{-1}R^*\\
&\quad +\frac{1}{6k^2}(\Delta'+k)^{-1}(2k-\Delta')(k-\Delta'-R^*)\\
&=\frac{1}{2k}+\frac{1}{2k}(\Delta'+k)^{-1}(k-\Delta'-R^*)+\frac{1}{2k}(\Delta'+k)^{-1}R^*\\
&=\frac{1}{2k}+\frac{1}{2k}(\Delta'+k)^{-1}(k-\Delta')\\
&=(\Delta'+k)^{-1},
	\end{align*}
which completes the proof. Here the second and last equalities follow from 
	\begin{align*}
	(\Delta'+k)^{-1}(k-\Delta')+Id=	2k(\Delta'+k)^{-1}.
	\end{align*}
while the third and fourth equalities follow from 
 \begin{align*}
	(\Delta'+k)^{-1}(2k-\Delta')+Id=3k(\Delta'+k)^{-1}.
	\end{align*}

\end{proof}

We shall treat each term in the expansion using 
the following lemmas. 
We refer  \cite[Chapter VII]{Dem} for the calculus
on K\"a{}hler manifolds.

\begin{lemma}\label{lemma1}
	Let $(X,\omega)$ be a compact K\"ahler manifold and $(L^k,e^{-k\phi})$ be  a Hermitian line bundle over $X$ with $\sqrt{-1}\p\b{\p}\phi=\omega$. For any $\alpha\in A^{n-1,1}(X, L^k)$ it holds
	\begin{align*}
	(k-\n'^*\n'-R^*)\alpha=(dv^t)^*\n'(\n_{\frac{\p}{\p\b{v}^t}}\alpha),
	\end{align*}
where $\n'^*$ is the adjoint operator of the $(1,0)$-part $\n'$ of Chern connection, $(dv^t)^*=\phi^{\b{t}s}i_{\frac{\p}{\p v^s}}$, $R^*=R^{\b{t}j}_{i\b{l}}i_{\frac{\p}{\p v^j}}dv^i\wedge d\b{v}^l\wedge i_{\frac{\p}{\p\b{v}^t}}$, $ R^{\b{t}j}_{i\b{l}}=-\p_{i}\p_{\b{l}}\phi^{\b{t}j}+\phi^{\b{t}s}_{\b{l}}\phi^{\b{k}j}_i\phi_{s\b{k}}$ is the Chern curvature component of $(T^*_X,(\phi^{i\b{j}}))$, and $\n_{\frac{\p}{\p\b{v}^t}}=\frac{\p}{\p\b{v}^t}-\o{\Gamma^l_{tj}}d\b{v}^j\wedge i_{\frac{\p}{\p\b{v}^l}}$, $\Gamma^l_{tj}=\frac{\p\phi_{t\b{k}}}{\p v^j}\phi^{\b{k}l}$. 
\end{lemma}
\begin{proof}
	By \cite[Chapter VII, Theorem (1.1)]{Dem},
        $\n'^*=\sqrt{-1}[\Lambda, \b{\p}]$,
 where $\Lambda$ is the adjoint of multiplication operator
        $\omega\wedge \bullet$
by the K\"ahler
        form.  Thus
\begin{align}\label{2.7}
\begin{split}
	\n'^*\n'\alpha
	&=\sqrt{-1}[\Lambda,\b{\p}]\n'\alpha\\
	&=\sqrt{-1}\Lambda\b{\p}\n'\alpha-\sqrt{-1}\b{\p}\Lambda\n'\alpha.
	\end{split}
\end{align}
We expand the second term and find
\begin{align}\label{2.8-1}
\begin{split}
&\quad \sqrt{-1}\b{\p}\Lambda\n'\alpha =\b{\p}\left(\phi^{s\b{t}}i_{\frac{\p}{\p\b{v}^t}}i_{\frac{\p}{\p v^s}}\n'\alpha\right)\\
&=(\b{\p}\phi^{s\b{t}})\wedge i_{\frac{\p}{\p \b{v}^t}}i_{\frac{\p}{\p v^s}}\n'\alpha+\phi^{s\b{t}}i_{\frac{\p}{\p\b{v}^t}}i_{\frac{\p}{\p v^s}}\b{\p}\n'\alpha+\phi^{s\b{t}} \frac{\p}{\p\b{v}^t}i_{\frac{\p}{\p v^s}}\n'\alpha\\
&=(\b{\p}\phi^{s\b{t}})\wedge i_{\frac{\p}{\p v^s}}\n'i_{\frac{\p}{\p \b{v}^t}}\alpha+\sqrt{-1}\Lambda\b{\p}\n'\alpha+\phi^{s\b{t}}i_{\frac{\p}{\p v^s}}\n'\frac{\p\alpha}{\p \b{v}^t}-k\alpha,
\end{split}
\end{align}
where the first equality holds since $[\b{\p}, i_{\frac{\p}{\p\b{v}^t}}]=\frac{\p}{\p\b{v}^t}$ and $[\b{\p},i_{\frac{\p}{\p v^i}}]=0$,
the second equality follows from $[\frac{\p}{\p
  v^t},\n']=-k\p\phi_{\b{t}}$ and $[i_{\frac{\p}{\p \b{v}^t}},
i_{\frac{\p}{\p v^s}}]=[i_{\frac{\p}{\p \b{v}^t}}, \n']=0$. Combining
(\ref{2.7}) with (\ref{2.8-1}) we obtain
\begin{align}\label{2.8}
(k-\n'^*\n)\alpha=(\b{\p}\phi^{s\b{t}})\wedge i_{\frac{\p}{\p v^s}}\n'i_{\frac{\p}{\p \b{v}^t}}\alpha+\phi^{s\b{t}}i_{\frac{\p}{\p v^s}}\n'\frac{\p\alpha}{\p \b{v}^t}.	
\end{align}
Furthermore the action on $\alpha$ of the second term in the RHS of
(\ref{2.8}),
by the definition of $\n_{\frac{\p}{\p\b{v}^t}}$, 
is the operator
\begin{align}\label{2.9}
	\begin{split}
		&\quad\phi^{s\b{t}}i_{\frac{\p}{\p v^s}}\n'\frac{\p}{\p \b{v}^t}\\
		&=\phi^{s\b{t}}i_{\frac{\p}{\p v^s}}\n'\n_{\frac{\p}{\p \b{v}^t}}+\phi^{s\b{t}}i_{\frac{\p}{\p v^s}}\n'(\o{\Gamma^k_{tl}}d\b{v}^l\wedge i_{\frac{\p}{\p\b{v}^k}})\\
		&=(dv^t)^*\n'\n_{\frac{\p}{\p
                    \b{v}^t}}+\phi^{s\b{t}}(\p_i\o{\Gamma^k_{tl}})
                i_{\frac{\p}{\p v^s}}dv^i\wedge d\b{v}^l\wedge
                i_{\frac{\p}{\p \b{v}^k}}+\phi^{s\b{t}}  \overline{\Gamma^k_{tl}} d\b{v}^l\wedge i_{\frac{\p}{\p v^s}}\n' i_{\frac{\p}{\p \b{v}^k}}\\
		&=(dv^t)^*\n'\n_{\frac{\p}{\p \b{v}^t}}+R^*-(\b{\p}\phi^{s\b{t}})\wedge i_{\frac{\p}{\p v^s}}\n'i_{\frac{\p}{\p \b{v}^t}}.
	\end{split}
\end{align}
Substituting (\ref{2.9}) into (\ref{2.8}) we have 
\begin{align*}
	(k-\n'^*\n'-R^*)\alpha=(dv^t)^*\n'(\n_{\frac{\p}{\p\b{v}^t}}\alpha).
	\end{align*}
\end{proof}
\begin{lemma}\label{lemma4} The following identity holds,
	$$\langle(k-\Delta')i_{\mu}u,i_{\mu}u\rangle=\int_{\mc{X}_y}(|\mu|^2_{R^*}-|\o{\n}\mu|^2)|u|^2e^{-k\phi},$$
	where  we have introduced $|\mu|^2_{R^*}:=\mu^l_{\b{t}}
        R^{\b{t}j}_{i\b{l}}\o{\mu^i_{\b{j}}}$, which need not to be
        nonnegative, $|\o{\n}\mu|^2=
\n_{\b{l}}\mu^i_{\b{t}}\o{\n_{\b{k}}\mu^s_{\b{q}}}\phi_{i\b{s}}\phi^{\b{t}q}\phi^{\b{l}k}$, $\n_{\b{l}}\mu^i_{\b{t}}=\p_{\b{l}}\mu_{\b{t}}^i-\o{\Gamma^s_{tl}}\mu^i_{\b{s}}$.
\end{lemma}
\begin{proof}
By a direct computation, we find
\begin{align*}
i_{\mu}u &=(-\p_{\b{l}}(\phi^{i\b{j}}\phi_{\alpha\b{j}})\zeta^{\alpha})d\b{v}^l\wedge i_{\frac{\p}{\p v^i}}(u'dv\otimes e^k)\\
&=\p_{\b{l}}(\phi^{i\b{j}}\phi_{\alpha\b{j}})\zeta^{\alpha}u'd\b{v}^l\wedge(-1)^idv^1\wedge \cdots\widehat{dv^i}\cdots \wedge dv^n\otimes e^k\\
	&=\n'^*(\phi_{\alpha\b{j}}\zeta^{\alpha}u'd\b{v}^j\wedge dv\otimes e^k)\\
	&=\n'^*(\phi_{\alpha\b{j}}\zeta^{\alpha}d\b{v}^j\wedge u).
\end{align*}
It follows that
  \begin{align}\label{2.3}
  	\Delta'i_{\mu}u=(\n'^*\n'+\n'\n'^*) i_{\mu}u=\n'^*\n' i_{\mu}u.  \end{align}
Thus,  using Lemma \ref{lemma1} and (\ref{2.3}), we obtain
\begin{align}\label{2.11}
\begin{split}
\langle(k-\Delta')i_{\mu}u,i_{\mu}u\rangle &=\langle(k-\n'^*\n)i_{\mu}u,i_{\mu}u\rangle\\
&=\langle(dv^t)^*\n'(\n_{\frac{\p}{\p\b{v}^t}}i_{\mu}u)+R^*i_{\mu}u, i_{\mu}u\rangle\\
&=\langle R^*i_{\mu}u,i_{\mu}u\rangle+\langle \n_{\frac{\p}{\p\b{v}^t}}i_{\mu}u,\n'^*(dv^t\wedge i_{\mu}u)\rangle.	
\end{split}
\end{align}
In terms of local coordinates the first term is
\begin{align}\label{2.13}
\begin{split}
	\langle R^*i_{\mu}u,i_{\mu}u\rangle &= \langle R^{\b{t}j}_{i\b{l}}i_{\frac{\p}{\p v^j}}dv^i\wedge d\b{v}^l\wedge i_{\frac{\p}{\p\b{v}^t}}\mu^s_{\b{q}}d\b{v}^q\wedge i_{\frac{\p}{\p v^s}}u,  i_{\mu}u\rangle\\
	&=\langle R^{\b{t}j}_{i\b{l}}\mu^i_{\b{t}}d\b{v}^l\wedge i_{\frac{\p}{\p v^j}}u, \mu^s_{\b{k}}d\b{v}^k\wedge i_{\frac{\p}{\p v^s}}u\rangle\\
	&=\int_{\mc{X}_y}(R^{\b{t}j}_{i\b{l}}\mu^i_{\b{t}}\o{\mu^s_{\b{k}}}\phi^{\b{l}k}\phi_{j\b{s}})|u|^2 e^{-k\phi}\\
	&=\int_{\mc{X}_y}(R^{\b{t}j}_{i\b{l}}\mu^i_{\b{t}}\o{\mu^l_{\b{j}}})|u|^2e^{-k\phi}=\int_{\mc{X}_y}|\mu|^2_{R^*}|u|^2e^{-k\phi}
\end{split}	
\end{align}
where the fourth equality follows from the definition of $\mu$ (\ref{definition of mu}) and 
\begin{align}\label{2.12}
\begin{split}
	\o{\mu^s_{\b{k}}}\phi^{\b{l}k}\phi_{j\b{s}}&=-\o{\zeta^{\alpha}}\p_k(\phi_{\b{\alpha}i}\phi^{i\b{s}})\phi^{\b{l}k}\phi_{j\b{s}}\\
	&=-\o{\zeta^{\alpha}}(\phi_{\b{\alpha}ik}\phi^{i\b{s}}+\phi_{\b{\alpha}i}\phi^{i\b{s}}_k)\phi^{\b{l}k}\phi_{j\b{s}}\\
	&=-\o{\zeta^{\alpha}}\phi_{\b{\alpha}ik}\phi^{i\b{s}}\phi^{\b{l}k}\phi_{j\b{s}}+\o{\zeta^{\alpha}}\phi_{\b{\alpha}i}\phi^{i\b{s}}\phi_{kj\b{s}}\phi^{\b{l}k}\\
	&=-\o{\zeta^{\alpha}}\phi_{\b{\alpha}jk}\phi^{\b{l}k}-\o{\zeta^{\alpha}}\phi_{\b{\alpha}i}\phi^{i\b{s}}_j\phi_{k\b{s}}\phi^{\b{l}k}\\
	&=-\o{\zeta^{\alpha}}\phi_{\b{\alpha}ji}\phi^{\b{l}i}-\o{\zeta^{\alpha}}\phi_{\b{\alpha}i}\phi^{i\b{l}}_j\\
	&=-\o{\zeta^{\alpha}}\p_j(\phi_{\b{\alpha}i}\phi^{\b{l}i})=\o{\mu^l_{\b{j}}}.
	\end{split}
\end{align}
The second term in the RHS of (\ref{2.11}) is
\begin{align}\label{2.14}
\begin{split}
	&\quad\langle \n_{\frac{\p}{\p\b{v}^t}}i_{\mu}u,\n'^*(dv^t\wedge i_{\mu}u)\rangle\\
	&=\langle(\n_{\b{t}}\mu^i_{\b{j}})d\b{v}^j\wedge i_{\frac{\p}{\p v^i}}u,-\b{\p}(\phi^{s\b{l}}\mu^t_{\b{l}})i_{\frac{\p}{\p v^s}}u\rangle\\
	&=-\langle(\b{\p}(\phi^{s\b{l}}\mu^t_{\b{l}})i_{\frac{\p}{\p v^s}})^*(\n_{\b{t}}\mu^i_{\b{j}})d\b{v}^j\wedge i_{\frac{\p}{\p v^i}}u,u\rangle\\
	&=-\langle\phi_{i\b{s}}\p_k(\phi^{\b{s}l}\o{\mu^t_{\b{l}}})\phi^{\b{j}k}(\n_{\b{t}}\mu^i_{\b{j}})u,u\rangle\\
	&=\int_{\mc{X}_y}-\n_k\o{\mu^t_{\b{i}}}\phi^{\b{j}k}\n_{\b{t}}\mu^i_{\b{j}}|u|^2e^{-k\phi}\\
	&=\int_{\mc{X}_y}-\n_k\o{\mu^t_{\b{i}}}\phi^{\b{j}k}\n_{\b{j}}\mu^i_{\b{t}}|u|^2e^{-k\phi}\\
	&=\int_{\mc{X}_y}-\n_k\o{\mu^t_{\b{i}}}\phi^{\b{j}k}\n_{\b{j}}\mu^s_{\b{q}}\phi^{\b{q}i}\phi_{s\b{t}}|u|^2e^{-k\phi}=\int_{\mc{X}_y}-|\o{\n}\mu|^2|u|^2e^{-k\phi},
\end{split}	
\end{align}
where in the second equality, $(\bullet)^*$ denotes the adjoint operator of $\bullet$, the fifth equality follows from 
\begin{align*}
	\n_{\b{t}}\mu^i_{\b{j}}=\n_{\b{t}}(\p_{\b{j}}(\phi_{\alpha\b{l}}\phi^{\b{l}i})\zeta^{\alpha})=\p_{\b{t}}\p_{\b{j}}(\phi_{\alpha\b{l}}\phi^{\b{l}i})\zeta^{\alpha}-\o{\Gamma^s_{tj}}\mu^i_{\b{s}}=\n_{\b{j}}\mu^i_{\b{t}}.
\end{align*}
The sixth equality holds by (\ref{2.12}) and $\n_{\b{j}}(\phi_{s\b{t}})=\n_{\b{j}}(\phi^{\b{q}i})=0$.
Substituting (\ref{2.13}) and (\ref{2.14}) into (\ref{2.11}) we have 
	$$\langle(k-\Delta')i_{\mu}u,i_{\mu}u\rangle=\int_{\mc{X}_y}(|\mu|^2_{R^*}-|\o{\n}\mu|^2)|u|^2e^{-k\phi}.$$
\end{proof}

In the subsequent text we shall
write  $O(k^{j})$ for any term that is of the order $k^j$ and is independent
of $u$.

\begin{lemma}\label{lemma2} We have the following expansion
$$\langle(k-\Delta')^2i_{\mu}u,i_{\mu}u\rangle=\int_{\mc{X}_y}(k|\o{\n}\mu|^2+O(1))|u|^2e^{-k\phi}.$$
\end{lemma}
\begin{proof} Writing $k-\Delta'=(k-\Delta'- R^*) + R^*$ and using Lemma \ref{lemma1} we have
	\begin{align}\label{2.15}
	\begin{split}
	&\quad \langle(k-\Delta')^2i_{\mu}u,i_{\mu}u\rangle =\|(k-\Delta')i_{\mu}u\|^2\\
	&=\|\phi^{s\b{t}}i_{\frac{\p}{\p v^s}}\n'(\n_{\frac{\p}{\p\b{v}^t}}i_{\mu}u)+R^*\alpha\|^2\\
	&=\|\phi^{s\b{t}}i_{\frac{\p}{\p v^s}}\n'(\n_{\frac{\p}{\p\b{v}^t}}i_{\mu}u)\|^2+\|R^*i_{\mu}u\|^2+2\text{Re}\langle\phi^{s\b{t}}i_{\frac{\p}{\p v^s}}\n'(\n_{\frac{\p}{\p\b{v}^t}}i_{\mu}u),R^*i_{\mu}u\rangle\\
	&=\langle\n'(\n_{\frac{\p}{\p\b{v}^j}}i_{\mu}u), dv^j\wedge\phi^{s\b{t}}i_{\frac{\p}{\p v^s}}\n'(\n_{\frac{\p}{\p\b{v}^t}}i_{\mu}u)\rangle+\|R^*i_{\mu}u\|^2\\
	&\quad+2\text{Re}\langle\n'(\n_{\frac{\p}{\p\b{v}^t}}i_{\mu}u),dv^t\wedge R^*i_{\mu}u\rangle\\
	&=\langle\n_{\frac{\p}{\p\b{v}^j}}i_{\mu}u, \n'^*(\phi^{j\b{t}}\n'(\n_{\frac{\p}{\p\b{v}^t}}i_{\mu}u))\rangle+\|R^*i_{\mu}u\|^2+2\text{Re}\langle\n_{\frac{\p}{\p\b{v}^t}}i_{\mu}u,\n'^*(dv^t\wedge R^*i_{\mu}u)\rangle.\\
	\end{split}
	\end{align}
For the first term in the RHS of (\ref{2.15}), we have
\begin{align}\label{2.16}
\begin{split}
	&\quad\langle\n_{\frac{\p}{\p\b{v}^j}}i_{\mu}u, \n'^*(\phi^{j\b{t}}\n'(\n_{\frac{\p}{\p\b{v}^t}}i_{\mu}u))\rangle \\
	&=\langle\n_{\frac{\p}{\p\b{v}^j}}i_{\mu}u, (\phi^{j\b{t}}\n'^*-\phi^{s\b{l}}\phi^{j\b{t}}_{\b{l}}i_{\frac{\p}{\p v^s}})\n'(\n_{\frac{\p}{\p\b{v}^t}}i_{\mu}u)\rangle\\
	&=-\langle\n'^*(\phi^{\b{j}t}_l dv^l\wedge \n_{\frac{\p}{\p\b{v}^j}}i_{\mu}u),\n_{\frac{\p}{\p\b{v}^t}}i_{\mu}u\rangle+\langle\n_{\frac{\p}{\p\b{v}^j}}i_{\mu}u, \phi^{j\b{t}}\n'^*\n'(\n_{\frac{\p}{\p\b{v}^t}}i_{\mu}u)\rangle.
	\end{split}
\end{align}
For the second term in RHS of (\ref{2.16}), using Lemma \ref{lemma1}, we obtain
\begin{align}\label{2.17}
\begin{split}
	&\quad\langle\n_{\frac{\p}{\p\b{v}^j}}i_{\mu}u, \phi^{j\b{t}}\n'^*\n'(\n_{\frac{\p}{\p\b{v}^t}}i_{\mu}u)\rangle \\
	&=\langle\n_{\frac{\p}{\p\b{v}^j}}i_{\mu}u, \phi^{j\b{t}}(k-\phi^{s\b{l}}i_{\frac{\p}{\p v^s}}\n'\n_{\frac{\p}{\p\b{v}^l}}-R^*)(\n_{\frac{\p}{\p\b{v}^t}}i_{\mu}u)\rangle\\
	&=\int_{\mc{X}_y}k\phi^{\b{j}t}\phi^{\b{l}s}\phi_{k\b{i}}\n_{\b{j}}\mu^k_{\b{l}}\n_t\o{\mu^i_{\b{s}}}|u|^2e^{-\phi}-\langle\n'^*(dv^l\wedge \n_{\frac{\p}{\p\b{v}^j}}i_{\mu}u),\n_{\frac{\p}{\p\b{v}^j}}\n_{\frac{\p}{\p\b{v}^t}}i_{\mu}u\rangle\\
	&\quad -\langle\n_{\frac{\p}{\p\b{v}^j}}i_{\mu}u, \phi^{j\b{t}}R^*(\n_{\frac{\p}{\p\b{v}^t}}i_{\mu}u)\rangle.
	\end{split}
\end{align}
We substitute (\ref{2.16}) and  (\ref{2.17}) into (\ref{2.15}), 
\begin{align*}
&\quad \langle(k-\Delta')^2i_{\mu}u,i_{\mu}u\rangle =\|R^*i_{\mu}u\|^2+2\text{Re}\langle\n_{\frac{\p}{\p\b{v}^t}}i_{\mu}u,\n'^*(dv^t\wedge R^*i_{\mu}u)\rangle	\\
&-\langle\n'^*(\phi^{\b{j}t}_l dv^l\wedge \n_{\frac{\p}{\p\b{v}^j}}i_{\mu}u),\n_{\frac{\p}{\p\b{v}^t}}i_{\mu}u\rangle
-\langle\n'^*(dv^l\wedge \n_{\frac{\p}{\p\b{v}^j}}i_{\mu}u),\n_{\frac{\p}{\p\b{v}^j}}\n_{\frac{\p}{\p\b{v}^t}}i_{\mu}u\rangle\\
&-\langle\n_{\frac{\p}{\p\b{v}^j}}i_{\mu}u, \phi^{j\b{t}}R^*(\n_{\frac{\p}{\p\b{v}^t}}i_{\mu}u)\rangle+\int_{\mc{X}_y}k\phi^{\b{j}t}\phi^{\b{l}s}\phi_{k\b{i}}\n_{\b{j}}\mu^k_{\b{l}}\n_t\o{\mu^i_{\b{s}}}|u|^2e^{-k\phi}\\
&=\int_{\mc{X}_y}(k|\o{\n}\mu|^2+O(1))|u|^2e^{-k\phi}.
\end{align*}
This completes the proof.
\end{proof}
 
\begin{lemma}\label{lemma3} The quadratic form $\langle (k-\Delta')^3i_{\mu}u, i_{\mu}u\rangle $ has 
the following expansion
$$\langle (k-\Delta')^3i_{\mu}u, i_{\mu}u\rangle=\int_{\mc{X}_y}(-k^2|\o{\n}u|^2+O(k))|u|^2e^{-k\phi}.$$
\end{lemma}
\begin{proof}
Denote $\n_{\b{t}}=\n_{\frac{\p}{\p\b{v}^t}}$.  We have, using Lemma \ref{lemma1}, that
\begin{align}\label{2.18}
\begin{split}
&\quad \langle(k-\Delta')^3 i_{\mu}u,i_{\mu}u\rangle	\\
&=\langle((dv^t)^*\n'\n_{\b{t}}+R^*)((dv^l)^*\n'\n_{\b{l}}+R^*)((dv^s)^*\n'\n_{\b{s}}+R^*)i_{\mu}u,i_{\mu}u\rangle\\
&=O(k)+\langle(dv^t)^*\n'\n_{\b{t}}(dv^l)^*\n'\n_{\b{l}}(dv^s)^*\n'\n_{\b{s}}i_{\mu}u,i_{\mu}u\rangle\\
&=O(k)+\langle(dv^t)^*\n'(dv^l)^*\n_{\b{t}}\n'(dv^s)^*\n_{\b{l}}\n'\n_{\b{s}}i_{\mu}u,i_{\mu}u\rangle\\
&=O(k)+k^2\langle (dv^t)^*\n'(dv^l)^*\p\phi_{\b{t}}\wedge (dv^s)^*\p\phi_{\b{l}}\wedge \n_{\b{s}}i_{\mu}u,i_{\mu}u\rangle\\
&=O(k)+k^2\langle\p\phi_{\b{t}}\wedge (dv^s)^*\p\phi_{\b{l}}\wedge \n_{\b{s}}i_{\mu}u, dv^l\wedge \n'^*(dv^t\wedge i_{\mu}u)\rangle\\
&=O(k)+k^2\langle-\phi_{i\b{l}}(\n_{\b{t}}\mu^i_{\b{j}})d\b{v}^j u, \b{\p}(\phi^{l\b{m}}\mu^t_{\b{m}})u\rangle\\
&=\int_{\mc{X}_y}(-k^2|\o{\n}\mu|^2+O(k))|u|^2e^{-k\phi},
\end{split}
\end{align} 
	where the second equality holds because all the terms  in $$\langle((dv^t)^*\n'\n_{\b{t}}+R^*)((dv^l)^*\n'\n_{\b{l}}+R^*)((dv^s)^*\n'\n_{\b{s}}+R^*)i_{\mu}u,i_{\mu}u\rangle$$ containing the factor $R^*$ are treated
similarly as $\langle(k-\Delta')^2i_{\mu}u,i_{\mu}u\rangle$, which are in $O(k)$, the third equality holds since 
	$[\n_{\b{t}}, (dv^l)^*]=\p_{\b{t}}(\phi^{\b{l}i})i_{\frac{\p}{\p v^i}}$, and so its adjoint operator is in $O(1)$, the fourth equality follows from $[\n_{\b{t}}, \n']=k\p\phi_{\b{t}}$. 
\end{proof}

\begin{lemma}\label{lemma5} The  following expansion holds
$$\langle(k-\Delta')^4 i_{\mu}u,i_{\mu}u\rangle=\int_{\mc{X}_y}(k^3|\o{\n}\mu|^2+O(k^2))|u|^2e^{-k\phi}.$$	
\end{lemma}
\begin{proof}
	Similar to the proof in Lemma \ref{lemma3}
for  estimating reminder terms we have
	\begin{align}
	\begin{split}
		&\quad \langle(k-\Delta')^4 i_{\mu}u,i_{\mu}u\rangle	\\
&=\langle((dv^q)^*\n'\n_{\b{q}}+R^*)((dv^t)^*\n'\n_{\b{t}}+R^*)((dv^l)^*\n'\n_{\b{l}}+R^*)((dv^s)^*\n'\n_{\b{s}}+R^*)i_{\mu}u,i_{\mu}u\rangle\\
&=O(k^2)+\langle(dv^q)^*\n'\n_{\b{q}}(dv^t)^*\n'\n_{\b{t}}(dv^l)^*\n'\n_{\b{l}}(dv^s)^*\n'\n_{\b{s}}i_{\mu}u,i_{\mu}u\rangle\\
&=O(k^2)+\langle(dv^q)^*\n'(dv^t)^*\n_{\b{q}}\n'(dv^l)^*\n_{\b{t}}\n'(dv^s)^*\n_{\b{l}}\n'\n_{\b{s}}i_{\mu}u,i_{\mu}u\rangle\\
&=O(k^2)-k^3\langle (dv^q)^*\n'(dv^t)^*\p\phi_{\b{q}}\wedge(dv^l)^*\p\phi_{\b{t}}\wedge (dv^s)^*\p\phi_{\b{l}}\wedge \n_{\b{s}}i_{\mu}u,i_{\mu}u\rangle\\
&=O(k^2)-k^3\langle\p\phi_{\b{l}}\wedge\n_{\b{s}}i_{\mu}u, dv^s\wedge \n'^*(dv^l\wedge i_{\mu}u)\rangle\\
&=O(k^2)+k^3\langle\phi_{i\b{l}}(\n_{\b{s}}\mu^i_{\b{j}})d\b{v}^j u, \b{\p}(\phi^{s\b{m}}\mu^l_{\b{m}})u\rangle\\
&=\int_{\mc{X}_y}(k^3|\o{\n}\mu|^2+O(k^2))|u|^2e^{-k\phi}.
	\end{split}
	\end{align}
\end{proof}

We prove now Theorem \ref{main theorem 1}.
\begin{proof}
The curvature formula in (\ref{L2 curvature}) contains two quadratic forms in $u$. We treat first
 the second one defined by the resolvent $(\Delta '+k)^{-1}
$.
 We have, by the expansion (\ref{2.20})
$$
\langle (\Delta '+k)^{-1}i_\mu u, i_\mu u
\rangle 
=I(u) + II(u) +\cdots +  VII(u)
$$
$$
I(u)=
\langle I i_\mu u, i_\mu u\rangle, \qquad
II(u)=
\langle II i_\mu u, i_\mu u\rangle, \cdots, 
VII(u)=
\langle VII i_\mu u, i_\mu u\rangle.
$$

We shall treat each term using the lemmas above. 
First we have
\begin{align}\label{2.26}
I(u)=\langle\frac{1}{2k}i_{\mu}u,i_{\mu}u\rangle=\frac{1}{2k}\int_{\mc{X}_y}(\mu^i_{\b{j}}\o{\mu^s_{\b{t}}}\phi_{i\b{s}}\phi^{\b{j}t})|u|^2e^{-\phi}=\frac{1}{2k}\int_{\mc{X}_y}|\mu|^2|u|^2e^{-k\phi},	
\end{align}
where $|\mu|^2=\mu^i_{\b{j}}\o{\mu^s_{\b{t}}}\phi_{i\b{s}}\phi^{\b{j}t}$.

By Lemma \ref{lemma4} and (\ref{2.13}), the second term 
is 
\begin{align}\label{2.22}
II(u)=-\frac{1}{6k^2}\int_{\mc{X}_y}|\o{\n}\mu|^2|u|^2e^{-k\phi}.	
\end{align}
Likewise, by (\ref{2.13})
\begin{align}\label{2.27}
III(u)=\frac{1}{4k^2}\int_{\mc{X}_y}|\mu|^2_{R^*}|u|^2e^{-k\phi}.
\end{align}

The fourth term is
\begin{align}\label{2.28}
\begin{split}
IV(u)&\leq\frac{1}{4k^2}\|R^*i_{\mu}u\|\cdot\|(\Delta'+k)^{-1}(k-\Delta')i_{\mu}u\|\\
	&\leq \frac{1}{4k^3}\|R^*i_{\mu}u\|\cdot\langle(k-\Delta')^2i_{\mu}u,i_{\mu}u\rangle^{\frac{1}{2}}\\
	&=\frac{1}{4k^3}(O(1)\|\mu\|)\cdot(O(k^{\frac{1}{2}})\|u\|)\\
	&=O\left(\frac{1}{k^{\frac{5}{2}}}\right)\|u\|^2,
\end{split}	
\end{align}
where the third equality follows from Lemma \ref{lemma2}.

By Lemma \ref{lemma1}, Lemma \ref{lemma4} and Lemma \ref{lemma2}, 
the fifith term 
\begin{align}\label{2.21}
\begin{split}
	 V(u)
&=\frac{1}{18k^2}\langle(k-\Delta'-R^*)i_{\mu}u,i_{\mu}u\rangle+\frac{1}{18k^3}(\langle(k-\Delta')^2i_{\mu}u,i_{\mu}u\rangle\\
	&\quad-\langle(k-\Delta')R^*i_{\mu}u,i_{\mu}u\rangle)\\
	&=-\frac{1}{18k^2}\int_{\mc{X}_y}|\o{\n}\mu|^2|u|^2e^{-k\phi}+\frac{1}{18k^3}\int_{\mc{X}_y}(k|\o{\n}\mu|^2+O(1))|u|^2e^{-k\phi}\\
	&\quad-\frac{1}{18k^3}(\langle(R^*)^2i_{\mu}u,i_{\mu}u\rangle+\langle\n_{\frac{\p}{\p \b{v}^t}}R^*i_{\mu}u,\n'^*(dv^t\wedge i_{\mu}u)\rangle)\\
	&=O\left(\frac{1}{k^3}\right)\|u\|^2.
\end{split}	
\end{align}
By (\ref{2.22}), (\ref{2.21}), Lemma \ref{lemma3} and Lemma \ref{lemma2},
 the sixth term becomes 
\begin{align}\label{2.29}
\begin{split}
	VI(u)&=\langle\frac{1}{36k^4}(k^2+2k(k-\Delta)+(k-\Delta')^2)(k-\Delta'-R^*)i_{\mu}u,i_{\mu}u\rangle\\
	&=\frac{1}{36k^2}\langle(k-\Delta'-R^*)i_{\mu}u,i_{\mu}u\rangle+\frac{1}{18k^3}\langle(k-\Delta')(k-\Delta'-R^*)i_{\mu}u,i_{\mu}u\rangle\\
	&\quad+\frac{1}{36k^4}\langle(k-\Delta')^3i_{\mu}u,i_{\mu}u\rangle-\frac{1}{36k^4}\langle(k-\Delta')^2 R^*i_{\mu}u,i_{\mu}u\rangle\\
	&=-\frac{1}{36k^2}\langle(k-\Delta'-R^*)i_{\mu}u,i_{\mu}u\rangle+\frac{1}{18k^3}\langle(2k-\Delta')(k-\Delta'-R^*)i_{\mu}u,i_{\mu}u\rangle\\
	&\quad+\frac{1}{36k^4}\langle(k-\Delta')^3i_{\mu}u,i_{\mu}u\rangle-\frac{1}{36k^4}\langle(k-\Delta')^2 R^*i_{\mu}u,i_{\mu}u\rangle\\
	&=O\left(\frac{1}{k^3}\right)\|u\|^2.
\end{split}	
\end{align}
Here we have used Lemmas \ref{lemma4} and \ref{lemma3} to conclude
\begin{align*}
	&\quad -\frac{1}{36k^2}\langle(k-\Delta'-R^*)i_{\mu}u,i_{\mu}u\rangle+\frac{1}{36k^4}\langle(k-\Delta')^3i_{\mu}u,i_{\mu}u\rangle\\
	&=-\frac{1}{36k^2}\int_{\mc{X}_y}(-|\o{\n}\mu|^2)|u|^2e^{-k\phi}+\frac{1}{36k^4}\int_{\mc{X}_y}(-k^2|\o{\n}u|^2+O(k))|u|^2e^{-k\phi}\\
	&=\frac{1}{36k^4}\int_{\mc{X}_y}O(k)|u|^2e^{-k\phi}=O\left(\frac{1}{k^3}\right)\|u\|^2.
\end{align*}

Finally  the last term is
\begin{align}\label{2.23}
	\begin{split}	
VII(u)
	&=\frac{1}{36k^4}\langle(\Delta'+k)^{-1}(k-\Delta')(2k-\Delta')i_{\mu}u,(k-\Delta')(2k-\Delta')i_{\mu}u\rangle\\
	&\quad-\frac{1}{36k^4}\langle(2k-\Delta')^2R^*i_{\mu}u,(\Delta'+k)^{-1}(k-\Delta')i_{\mu}u\rangle.
	\end{split}
\end{align}
Note that 
\begin{align}\label{2.24}
	\begin{split}
		&\quad\langle\frac{1}{36k^4}(\Delta'+k)^{-1}(k-\Delta')(2k-\Delta')i_{\mu}u,(k-\Delta')(2k-\Delta')i_{\mu}u\rangle\\
		&\leq \frac{1}{36k^5}\langle(k-\Delta')(2k-\Delta')i_{\mu}u,(k-\Delta')(2k-\Delta')i_{\mu}u\rangle\\
		&=\frac{1}{36k^5}\langle(k-\Delta')^4i_{\mu}u,i_{\mu}u\rangle+\frac{1}{18k^4}\langle(k-\Delta')^3i_{\mu}u,i_{\mu}u\rangle\\
		&\quad+\frac{1}{36k^3}\langle(k-\Delta')^2i_{\mu}u,i_{\mu}u\rangle
		=O\left(\frac{1}{k^3}\right)\|u\|^2,
	\end{split}
\end{align}
where the last equality follows from Lemma \ref{lemma2}, Lemma \ref{lemma3} and Lemma \ref{lemma5}.

We estimate
 the second term in the RHS of (\ref{2.24}) as
\begin{align}\label{2.25}
\begin{split}
	&\quad |-\frac{1}{36k^4}\langle(2k-\Delta')^2R^*i_{\mu}u,(\Delta'+k)^{-1}(k-\Delta')i_{\mu}u\rangle|\\
	&\leq \frac{1}{36k^4}\|(2k-\Delta')^2R^*i_{\mu}u\|\cdot\|(\Delta'+k)^{-1}(k-\Delta')i_{\mu}u\|\\
	&\leq \frac{1}{36k^5}\|(2k-\Delta')^2R^*i_{\mu}u\|\cdot\langle(k-\Delta')^2i_{\mu}u, i_{\mu}u\rangle^{1/2}\\
	&= \frac{1}{36k^5}(O(k^2)\|u\|)\cdot (O(k^{\frac{1}{2}})\|u\|)\\
	&=O\left(\frac{1}{k^{\frac{5}{2}}}\right)\|u\|^2;
	\end{split}
\end{align}
indeed the
third equality follows from Lemma \ref{lemma2} and the following equality
\begin{align}
\begin{split}
	&\quad \|(2k-\Delta')^2R^*i_{\mu}u\|=|\langle(2k-\Delta')^4R^*i_{\mu}u, R^*i_{\mu}u\rangle|^{1/2}\\
	&=|\langle \left(k^4+4k^3(k-\Delta')+6k^2(k-\Delta')^2+4k(k-\Delta')^3+(k-\Delta')^4\right)i_{\tilde{\mu}}u, i_{\tilde{\mu}}u\rangle|^{1/2}\\
	&=O(k^2)\|u\|,
	\end{split}
\end{align}
where  $\tilde{\mu}= R^{\b{t}j}_{i\b{l}}\mu^i_{\b{t}}d\b{v}^l \otimes \frac{\p}{\p v^j}$ being the contraction of $\mu$ agains the curvature tensor $R$.

Substituting
(\ref{2.24}) and (\ref{2.25}) into (\ref{2.23}) results in
\begin{align}\label{2.30}
\langle\frac{1}{36k^4}(\Delta'+k)^{-1}(k-\Delta')(2k-\Delta')^2(k-\Delta'-R^*)i_{\mu}u,i_{\mu}u\rangle=O\left(\frac{1}{k^{\frac{5}{2}}}\right)\|u\|^2.	
\end{align}

Putting the quantities (\ref{2.26}), (\ref{2.22}), (\ref{2.27}), (\ref{2.28}), (\ref{2.21}), (\ref{2.29}), (\ref{2.30}) into (\ref{2.20}), we obtain
\begin{align}\label{2.31}
\begin{split}
&\quad \langle(\Delta'+k)^{-1}i_{\mu}u, i_{\mu}u\rangle\\
&=\int_{\mc{X}_y}\left(\frac{1}{2k}|\mu|^2+\left(-\frac{1}{6}|\o{\n}\mu|^2+\frac{1}{4}|\mu|^2_{R^*}\right)\frac{1}{k^2}+o(k^{-2})\right)|u|^2e^{-k\phi}.	
\end{split}
\end{align}

Finally substituting (\ref{2.31}) into (\ref{2.1}) we get
\begin{align}\label{2.32}
\begin{split}
	&\quad \langle\Theta^{E^k}u,u\rangle(\zeta,\b{\zeta})=-\sqrt{-1}\left(\int_{\mc{X}_y}kc(\phi)|u|^2e^{-k\phi}\right)(\zeta,\b{\zeta})\\
	&\quad+\int_{\mc{X}_y}\left(\frac{1}{2}|\mu|^2+\left(-\frac{1}{6}|\o{\n}\mu|^2+\frac{1}{4}|\mu|^2_{R^*}\right)\frac{1}{k}+o(k^{-1})\right)|u|^2e^{-k\phi}.
	\end{split}
\end{align}

Denote $d_k=\dim H^0(\mc{X}_y, L^k+K_{\mc{X}/M})$, and let $\{u_j\}_{j=1}^{d_k}$ be an orthogonal basis of $H^0(\mc{X}_y, L^k+K_{\mc{X}/M})$.  From (\ref{Bergman expansion}) and (\ref{metric consist}), we have 
\begin{align}\label{2.33}
\begin{split}
&\quad \sum_{j=1}^{d_k}|u_j|^2 e^{-k\phi}=\sum_{j=1}^{d_k}|u_j|^2_{L^2}\frac{\omega^n}{n!}\\
&=\left(k^n-\frac{\rho}{2}k^{n-1}+\left(-\frac{1}{6}\Delta\rho+\frac{1}{24}(|R|^2-4|Ric|^2+3\rho^2)\right)k^{n-2}+o(k^{n-2})\right)\frac{c(L,\phi)^n}{n!},	
\end{split}
\end{align}
where $c(L,\phi)=\frac{1}{2\pi}\omega$ is the Chern form of the Hermitian line bundle $(L,e^{-\phi})$.

We take the trace to both sides of (\ref{2.32}) and use the Bergman kernel expansion (\ref{2.33}), 
\begin{align}\label{2.34}
\begin{split}
	&\quad -\sqrt{-1}c_1(E^k,\|\bullet\|_k)(\zeta,\b{\zeta})=\frac{1}{2\pi}tr\Theta^{E_k}(\zeta,\b{\zeta})\\
	&=\frac{1}{2\pi}\int_{\mc{X}_y}\left(-\sqrt{-1}kc(\phi)(\zeta,\b{\zeta})+\frac{1}{2}|\mu|^2+\left(-\frac{1}{6}|\o{\n}\mu|^2+\frac{1}{4}|\mu|^2_{R^*}\right)\frac{1}{k}+o(k^{-1})\right)\\
	&\quad \cdot\left(k^n-\frac{\rho}{2}k^{n-1}+\left(-\frac{1}{6}\Delta\rho+\frac{1}{24}(|R|^2-4|Ric|^2+3\rho^2)\right)k^{n-2}+o(k^{n-2})\right)\frac{c(L,\phi)^n}{n!}\\
	&=\frac{k^{n+1}}{(2\pi)^{n+1}}\int_{\mc{X}_y}(-\sqrt{-1})c(\phi)(\zeta,\b{\zeta})\frac{\omega^n}{n!}+\frac{k^n}{(2\pi)^{n+1}}\int_{\mc{X}_y}\left(\frac{1}{2}|\mu|^2-\frac{\rho}{2}(-\sqrt{-1})c(\phi)(\zeta,\b{\zeta})\right)\frac{\omega^n}{n!}\\
	&\quad+\frac{k^{n-1}}{(2\pi)^{n+1}}\int_{\mc{X}_y}\left((-\sqrt{-1})c(\phi)(\zeta,\b{\zeta})\left(-\frac{1}{6}\Delta\rho+\frac{1}{24}(|R|^2-4|Ric|^2+3\rho^2)\right)-\frac{\rho}{4}|\mu|^2\right)\frac{\omega^n}{n!}\\
	&\quad +\frac{k^{n-1}}{(2\pi)^{n+1}}\left(-\frac{1}{6}\|\o{\n}\mu\|^2+\frac{1}{4}\|\mu\|^2_{R^*}\right)+o(k^{n-1}),
\end{split}	
\end{align}
where we have denoted $\|\o{\n}\mu\|^2=\int_{\mc{X}_y}|\o{\n}\mu|^2\frac{\omega^n}{n!}$, $\|\mu\|^2_{R^*}=\int_{\mc{X}_y}|\mu|^2_{R^*}\frac{\omega^n}{n!}$.  
We rewrite the integrals $\|\mu\|^2_{R^*}$ and $\|\o{\n}\mu\|^2$ (of the anti-holomorphic connection 
$\o{\n}\mu$) in terms of $\partial^\ast\mu$ and  the  holomorphic connection $\nabla' \mu$.
 By Akizuki-Nakano identity \cite[Chapter VII, Corollary (1.3)]{Dem}, we have 
\begin{align*}
\|\b{\p}^*\mu\|^2&=\langle\Delta''\mu,\mu\rangle\\
&=\langle\Delta'\mu,\mu\rangle+\langle[\sqrt{-1}R,\Lambda]\mu,\mu\rangle\\
&=\|\n'\mu\|^2-\langle\sqrt{-1}\Lambda R\mu,\mu\rangle\\
&=\|\n'\mu\|^2-\langle(\phi^{\b{s}t}\mu^i_{\b{j}}R^k_{is\b{t}}d\b{v}^j-\phi^{s\b{t}}\mu^i_{\b{t}}R^k_{is\b{j}}d\b{v}^j)\otimes \frac{\p}{\p v^k}, \mu\rangle\\
&=\|\n'\mu\|^2-\langle(\mu^i_{\b{j}}R_{i\b{l}}\phi^{\b{l}k}d\b{v}^j-\phi^{s\b{t}}\mu^i_{\b{t}}R^k_{is\b{j}}d\b{v}^j)\otimes \frac{\p}{\p v^k}, \mu\rangle\\
&=\|\n'\mu\|^2-\int_{\mc{X}_y}(\mu^i_{\b{j}}\o{\mu^s_{\b{t}}}R_{i\b{l}}\phi^{\b{l}k}\phi^{\b{j}t}\phi_{k\b{s}})\frac{\omega^n}{n!}+\int_{\mc{X}_y}\phi^{s\b{t}}R^{k}_{is\b{j}}\mu^i_{\b{t}}\o{\mu^p_{\b{q}}}\phi^{\b{j}q}\phi_{k\b{p}}\frac{\omega^n}{n!}\\
&=\|\n'\mu\|^2-\|\mu\|^2_{Ric}-\|\mu\|^2_{R^*},
\end{align*}
where we have denoted
$\|\mu\|^2_{Ric}:=\int_{\mc{X}_y}(\mu^i_{\b{j}}\o{\mu^s_{\b{t}}}R_{i\b{l}}\phi^{\b{l}k}\phi^{\b{j}t}\phi_{k\b{s}})\frac{\omega^n}{n!}$,
which again needs not to be nonnegative.
Therefore, 
\begin{align}\label{2.35}
\|\mu\|^2_{R^*}=\|\n'\mu\|^2-\|\mu\|^2_{Ric}-\|\b{\p}^*\mu\|^2.	
\end{align}
By its definition we find also
\begin{align}\label{2.36}
\begin{split}
	\|\o{\n}\mu\|^2&=\int_{\mc{X}_y}
\n_{\b{l}}\mu^i_{\b{j}}\phi^{\b{l}k}\n_k\o{\mu^j_{\b{i}}}
\frac{\omega^n}{n!}=-\int_{\mc{X}_y}\n_k\n_{\b{l}}\mu^i_{\b{j}}\phi^{\b{l}k}\o{\mu^j_{\b{i}}}\frac{\omega^n}{n!}\\
	&=\int_{\mc{X}_y}(\n_{\b{l}}\n_k-\n_k\n_{\b{l}})\mu^i_{\b{j}}\phi^{\b{l}k}\o{\mu^j_{\b{i}}}\frac{\omega^n}{n!}-\int_{\mc{X}_y}\n_{\b{l}}\n_k\mu^i_{\b{j}}\phi^{\b{l}k}\o{\mu^j_{\b{i}}}\frac{\omega^n}{n!}\\
	&=\int_{\mc{X}_y}(\p_{\b{l}}\Gamma^i_{ks}\mu^s_{\b{j}}+\p_k\o{\Gamma^s_{lj}}\mu^i_{\b{s}})\phi^{\b{l}k}\o{\mu^j_{\b{i}}}\frac{\omega^n}{n!}+\int_{\mc{X}_y}\n_k\mu^i_{\b{j}}\phi^{\b{l}k}\o{\n_{l}\mu^j_{\b{i}}}\frac{\omega^n}{n!}\\
	&=-2\|\mu\|^2_{Ric}+\|\n'\mu\|^2.
\end{split}	
\end{align}

Substituting (\ref{2.35}) and (\ref{2.36}) into (\ref{2.34}) we obtain
finally
\begin{align}\label{asymptotic 1}
\begin{split}
	&\quad -\sqrt{-1}c_1(E^k,\|\bullet\|_k)(\zeta,\b{\zeta})=\frac{k^{n+1}}{(2\pi)^{n+1}}\\
	&
\times
\int_{\mc{X}_y}(-\sqrt{-1})c(\phi)(\zeta,\b{\zeta})\frac{\omega^n}{n!}+\frac{k^n}{(2\pi)^{n+1}}\int_{\mc{X}_y}\left(\frac{1}{2}|\mu|^2-\frac{\rho}{2}(-\sqrt{-1})c(\phi)(\zeta,\b{\zeta})\right)\frac{\omega^n}{n!}\\
	&+\frac{k^{n-1}}{(2\pi)^{n+1}}\int_{\mc{X}_y}\left((-\sqrt{-1})c(\phi)(\zeta,\b{\zeta})\left(-\frac{1}{6}\Delta\rho+\frac{1}{24}(|R|^2-4|Ric|^2+3\rho^2)\right)-\frac{\rho}{4}|\mu|^2\right)\frac{\omega^n}{n!}\\
	&+\frac{k^{n-1}}{(2\pi)^{n+1}}\left(\frac{1}{12}\|\mu\|^2_{Ric}+\frac{1}{12}\|\n'\mu\|^2-\frac{1}{4}\|\b{\p}^*\mu\|^2\right)+o(k^{n-1}).
\end{split}	
\end{align}
This completes the proof of Theorem \ref{main theorem 1}.
\end{proof}

\begin{rem} 
It is a general fact \cite{Ma} that the $L^2$-curvature $c_1(E^k)$ above 
has an expansion in the integer powers of $k$, so that the
lower order term $o(k^{n-1})$ in our statement can be written as $O(k^{n-2})$. Indeed observe
that the fractional order $O(k^{n-1-\frac 12})$-terms in
the proof of Theorem \ref{main theorem 1}   are all
due to the estimate  $\Vert(\Delta' +k)^{-1}(k-\Delta')i_\mu u\Vert
\le C k^{-\frac 12 }\Vert u\Vert $. However  we can use again Lemma 
\ref{expanding} and prove that the traces of the quadratic forms 
involving $(\Delta' +k)^{-1}(k-\Delta')$ are actually of integer
order  instead of fractional order, e.g. the trace of the quadratic
form $IV(u)$ in the estimate
(\ref{2.28}) has an expansion of order $k^{-3 +n}$. It might be
interesting  to find a recursive formula for the coefficients of the expansion
$c_1(E^k)$ following our proof and using the Bergman kernel expansion.
\end{rem}

\subsection{The asymptotic of the curvature of Quillen metric}\label{Asy2}
We  compute now the asymptotic of the curvature of Qullien metric and prove Theorem \ref{main theorem 2}. 

By Theorem \ref{Bis} and Theorem \ref{Bis11},  we have
\begin{align}\label{Bisexp}
\begin{split}
	&\quad -c_1(\lambda,\|\bullet\|_Q)=\left\{\int_{\mc{X}/M} \text{td}(\mc{X}/M, (\sqrt{-1}\p\b{\p}\phi)_y))e^{\frac{1}{2\pi}(k\omega+\sqrt{-1}R^{K_{\mc{X}/M}})}\right\}^{(1,1)}\\
	&=\left\{\int_{\mc{X}/M}(1+\frac{1}{2}c_1(T_{\mc{X}/M},\phi)+\frac{1}{12}(c_1(T_{\mc{X}/M},\phi)^2+c_2(T_{\mc{X}/M},\phi))+\cdots)\right.\\
	&\quad\left.\cdot\sum_{i=0}^{\infty}\frac{(2\pi)^{-i}(k\omega+\sqrt{-1}R^{K_{\mc{X}/M}})^i}{i!}\right\}^{(1,1)}\\
&=\frac{k^{n+1}}{(2\pi)^{n+1}}\int_{\mc{X}/M}\frac{\omega^{n+1}}{(n+1)!}+\frac{k^n}{(2\pi)^{n+1}}\int_{\mc{X}/M}\frac{1}{2}(\sqrt{-1}R^{K_{\mc{X}/M}})\wedge \frac{\omega^n}{n!}\\
&\quad+\frac{k^{n-1}}{(2\pi)^{n+1}}\int_{\mc{X}/M}\frac{1}{12}(\sqrt{-1}R^{K_{\mc{X}/M}})^2\wedge\frac{\omega^{n-1}}{(n-1)!}\\
&\quad+\frac{k^{n-1}}{(2\pi)^{n-1}}\int_{\mc{X}/M}\frac{1}{12}c_2(T_{\mc{X}/M},\phi)\wedge\frac{\omega^{n-1}}{(n-1)!}+O(k^{n-2}),
\end{split}	
\end{align}
where $\text{td}$ is the Todd character forms, which has an expansion,
$$\text{td}(F,h)=1+\frac{1}{2}c_1(F,h)+\frac{1}{12}(c_1(F,h)^2+c_2(F,h))+\frac{1}{24}c_1(F,h)c_2(F,h)+\cdots$$
for any Hermitian vector bundle $(F,h)$, the second equality follows from $c_1(T_{\mc{X}/M},\phi)=-\frac{\sqrt{-1}}{2\pi}\p\b{\p}\log\det\phi=-\frac{\sqrt{-1}}{2\pi}R^{K_{\mc{X}/M}}$.

Now we consider the last term in the RHS of (\ref{Bisexp}),
\begin{align}\label{2.37}
\begin{split}
&\quad\int_{\mc{X}/M}c_2(T_{\mc{X}/M},\phi)\wedge\frac{\omega^{n-1}}{12(n-1)!}\\
&=\int_{\mc{X}/M}\frac{1}{2}\left(c_1(K_{\mc{X}/M},\phi)^2-\left(\frac{\sqrt{-1}}{2\pi}\right)^2tr(R\wedge R)\right)\wedge\frac{\omega^{n-1}}{12(n-1)!}\\
&=\left(\frac{1}{2\pi}\right)^{2}\int_{\mc{X}/M}\frac{1}{24}(\sqrt{-1}R^{K_{\mc{X}/M}})^2\wedge\frac{\omega^{n-1}}{(n-1)!}\\
&\quad-\left(\frac{1}{2\pi}\right)^{2}\int_{\mc{X}/M}\frac{1}{24}(\sqrt{-1})^2tr(R\wedge R)\wedge\frac{\omega^{n-1}}{(n-1)!},
\end{split}
\end{align}
where the curvature operator $R$  is defined by 
\begin{align}\label{2.38}
\begin{split}
R&=R^i_{j}\delta v^j\otimes i_{\frac{\p}{\p v^i}}\\
&=\left(R^i_{j\alpha\b{\beta}}dz^{\alpha}\wedge d\b{z}^{\beta}+R^i_{j\alpha\b{\beta}}dz^{\alpha}\wedge \delta\b{v}^l+R^i_{jk\b{\beta}}\delta v^k\wedge d\b{z}^\beta+R^i_{jk\b{l}}\delta v^k\wedge \delta\b{v}^l\right)\delta v^j\otimes i_{\frac{\p}{\p v^i}}.
\end{split}
\end{align}
Here the second and third curvature term is
\begin{align}\label{2.40}
R^i_{j\alpha\b{l}}=\n'_{j}(\mu_{\alpha})^i_{\b{l}}, \quad R^j_{ik\b{\beta}}=\o{\n'_t(\mu_{\beta})^s_{\b{k}}}\phi^{\b{t}j}\phi_{\b{s}i}.	
\end{align}
In fact, one can prove them in terms of normal coordinates, i.e. $\phi_{i\b{j}}=\delta_{ij}$, $\phi_{i\b{j}k}=0$ at a fix point, so
\begin{align*}
R^j_{ik\b{\beta}}&=i_{\frac{\delta}{\delta\b{z}^{\beta}}}i_{\frac{\p}{\p v^k}}\b{\p}(\p\phi_{i\b{l}}\phi^{\b{l}j})\\
&=-(\p_{\b{\beta}}-\phi_{\b{\beta}t}\phi^{t\b{s}}\p_{\b{s}})(\phi_{i\b{l}k}\phi^{\b{l}j})\\
&=-\phi_{i\b{j}k\b{\beta}}+\phi_{\b{\beta}t}\phi_{i\b{j}k\b{t}}	\\
&=\n'_j(-\phi_{\alpha\b{k}\b{l}}\phi^{\b{k}i}+\phi_{\alpha\b{k}}\phi_{k\b{i}\b{l}})\\
&=\n_j(-\p_{\b{l}}(\phi_{\alpha\b{k}}\phi^{\b{k}i}))=\n'_{j}(\mu_{\alpha})^i_{\b{l}},
\end{align*}
while second identity in (\ref{2.40}) holds similarly.

We compute  the second term in the RHS of (\ref{2.37}),
\begin{align}\label{2.39}
\begin{split}
&\quad(\sqrt{-1})^2\int_{\mc{X}/M}tr(R\wedge R)\frac{\omega^{n-1}}{(n-1)!}\\
&=\int_{\mc{X}/M}(2R^i_{j\alpha\b{\beta}}R^j_{ik\b{l}}\phi^{k\b{l}}
 -2R^i_{j\alpha\b{l}}R^j_{ik\b{\beta}}\phi^{k\b{l}})\frac{\omega^n}{n!}\sqrt{-1}dz^{\alpha}\wedge d\b{z}^{\beta}\\
&\quad +\int_{\mc{X}/M}R^i_{jk\b{l}}R^j_{is\b{t}}c(\phi)(\sqrt{-1}\delta v^k\wedge \delta \b{v}^l)\wedge (\sqrt{-1}\delta v^s\wedge \delta \b{v}^t)\wedge \frac{\omega^{n-1}}{(n-2)!}\\
&=2\int_{\mc{X}/M}R^i_{j\alpha\b{\beta}}R^j_{ik\b{l}}\phi^{k\b{l}}\frac{\omega^n}{n!}\sqrt{-1}dz^{\alpha}\wedge d\b{z}^{\beta}-2\langle\n'\mu_{\alpha},\n'\mu_{\beta}\rangle\sqrt{-1}dz^{\alpha}\wedge d\b{z}^{\beta}\\
&\quad+\int_{\mc{X}/M}(|Ric|^2-|R|^2)c(\phi)\frac{\omega^n}{n!},
\end{split}
\end{align}
where the second equality follows from  the fact
\begin{align}\label{2.53}
n(n-1)\alpha\wedge\beta\wedge \omega^{n-2}=\left(\Lambda\alpha\cdot\Lambda\beta	-\langle\alpha,\beta\rangle\right)\omega^n
\end{align}
for two real $(1,1)$-forms $\alpha$ and $\beta$.

\begin{lemma} The following identities hold
\begin{align}\label{le1}
\int_{\mc{X}/M}&(\sqrt{-1}R^{K_{\mc{X}/M}})\wedge \frac{\omega^n}{n!}=-\int_{\mc{X}/M}\rho c(\phi)\frac{\omega^n}{n!}+\langle\mu_{\alpha}, \mu_{\beta}\rangle\sqrt{-1}dz^{\alpha}\wedge d\b{z}^{\beta}
\end{align}
	and 
	\begin{align}\label{le2}
	\int_{\mc{X}/M}(\Delta\rho) c(\phi)\frac{\omega^n}{n!}=\left(-\int_{\mc{X}/M}(R^i_{j\alpha\b{\beta}}R^j_{ik\b{l}}\phi^{k\b{l}})\frac{\omega^n}{n!}-\langle \mu_{\alpha}, \mu_{\beta}\rangle_{Ric}\right)\sqrt{-1}dz^{\alpha}\wedge d\b{z}^{\beta},	
	\end{align}
	where $\langle \mu_{\alpha}, \mu_{\beta}\rangle_{Ric}:=\int_{\mc{X}/M}((\mu_{\alpha})^i_{\b{j}}\o{(\mu_{\beta})^s_{\b{t}}}R_{i\b{l}}\phi^{\b{l}k}\phi^{\b{j}t}\phi_{k\b{s}})\frac{\omega^n}{n!}$, which satisfies 
	\begin{align}\label{le4}
	\langle \mu_{\alpha}, \mu_{\beta}\rangle_{Ric}\zeta^{\alpha}\o{\zeta^{\beta}}=\|\mu\|^2_{Ric}.	
	\end{align}

\end{lemma}
\begin{proof}
For any fixed point $p\in\mc{X}_y$, $y\in M$, we take normal coordinates near $p$ such that $\phi_{i\b{j}}(p)=\delta_{ij}$, $\phi_{i\b{j}k}(p)=0$. Recall that $(\mu_{\alpha})^k_{\b{l}}=-\p_{\b{l}}(\phi_{\alpha\b{t}}\phi^{\b{t}k})$ and denote $c(\phi)_{\alpha\b{\beta}}=\phi_{\alpha\b{\beta}}-\phi_{\alpha\b{l}}\phi_{k\b{\beta}}\phi^{k\b{l}}$.  Evaluating at $p$
we see that
\begin{align}\label{le3}
\begin{split}
\p_{k}\p_{\b{l}}c(\phi)_{\alpha\b{\beta}}&=\p_k\p_{\b{l}}(\phi_{\alpha\b{\beta}}-\phi_{\alpha\b{j}}\phi^{\b{j}i}\phi_{i\b{\beta}})\\
&=\phi_{k\b{l}\alpha\b{\beta}}-\phi_{\alpha\b{j}k\b{l}}\phi_{j\b{\beta}}-\phi_{\alpha\b{j}}\phi_{i\b{\beta}k\b{l}}+\phi_{\alpha\b{j}}\phi_{k\b{l}j\b{i}}\phi_{i\b{\beta}}\\
&\quad-\phi_{\alpha\b{j}k}\phi_{j\b{\beta}\b{l}}-\phi_{\alpha\b{j}\b{l}}\phi_{j\b{\beta}k}\\
&=(-\b{\p}(\p\phi_{k\b{t}}\phi^{\b{t}s})\phi_{s\b{l}})(\frac{\delta}{\delta z^{\alpha}},\frac{\delta}{\delta \b{z}^{\beta}})-(\mu_{\alpha})^j_{\b{l}}\o{(\mu_{\beta})^s_{\b{j}}}\phi_{k\b{s}}\\
&=-R^s_{k\alpha\b{\beta}}\phi_{s\b{l}}-(\mu_{\alpha})^j_{\b{l}}\o{(\mu_{\beta})^s_{\b{j}}}\phi_{k\b{s}},
\end{split}
\end{align}
where the last equality follows from (\ref{2.38}). Consequentely
using  (\ref{2.12}) we get
	\begin{align}\label{2.54}
	\begin{split}
	\Delta c(\phi)_{\alpha\b{\beta}}&=\phi^{\b{l}k}\p_k\p_{\b{l}}c(\phi)_{\alpha\b{\beta}}\\
	&=-R^k_{k\alpha\b{\beta}}-(\mu_{\alpha})^j_{\b{l}}\o{(\mu_{\beta})^l_{\b{j}}}\\
	&=(\p\b{\p}\log\det\phi)(\frac{\delta}{\delta z^{\alpha}},\frac{\delta}{\delta\b{z}^{\beta}})-(\mu_{\alpha})^i_{\b{j}}\o{(\mu_{\beta})^t_{\b{s}}}\phi^{\b{j}s}\phi_{i\b{t}},
	\end{split}
	\end{align}
We perform the integration using Stoke's theorem
on  (\ref{le1}),
\begin{align*}
	\int_{\mc{X}/M}&(\sqrt{-1}R^{K_{\mc{X}/M}})\wedge \frac{\omega^n}{n!}=\int_{\mc{X}/M}\p\b{\p}\log\det\phi(\frac{\delta}{\delta z^{\alpha}},\frac{\delta}{\delta\b{z}^{\beta}})\frac{\omega^n}{n!}\sqrt{-1}dz^{\alpha}\wedge d\b{z}^{\beta}\\
	&\quad +\int_{\mc{X}/M}\p\b{\p}\log\det\phi(\frac{\delta}{\p v^{i}},\frac{\p}{\p\b{v}^{j}})\sqrt{-1}\delta v^i\wedge \delta\b{v}^j\wedge c(\phi)\frac{\omega^{n-1}}{(n-1)!}\\
	&=\int_{\mc{X}/M}\left((\mu_{\alpha})^i_{\b{j}}\o{(\mu_{\beta})^t_{\b{s}}}\phi^{\b{j}s}\phi_{i\b{t}}\sqrt{-1}dz^{\alpha}\wedge d\b{z}^{\beta}-\rho c(\phi)\right)\frac{\omega^n}{n!}\\
	&=-\int_{\mc{X}/M}\rho c(\phi)\frac{\omega^n}{n!}+\langle\mu_{\alpha}, \mu_{\beta}\rangle\sqrt{-1}dz^{\alpha}\wedge d\b{z}^{\beta},
\end{align*}
which proves (\ref{le1}).

On the other hand we have also  by (\ref{le3}) that
\begin{align*}
\begin{split}
	(\phi^{\b{l}s}R_{s\b{t}}\phi^{\b{t}k})\p_k\p_{\b{l}}c(\phi)_{\alpha\b{\beta}}&=\phi^{\b{l}i}R_{i\b{t}}\phi^{\b{t}k}(-R^s_{k\alpha\b{\beta}}\phi_{s\b{l}}-(\mu_{\alpha})^j_{\b{l}}\o{(\mu_{\beta})^s_{\b{j}}}\phi_{k\b{s}})\\
	&=-R_{i\b{t}}\phi^{\b{t}k}R^i_{k\alpha\b{\beta}}-\phi^{\b{l}i}R_{i\b{t}}(\mu_{\alpha})^i_{\b{l}}\o{(\mu_{\beta})^t_{\b{j}}}\\
	&=-R^i_{j\alpha\b{\beta}}R^j_{ik\b{l}}\phi^{k\b{l}}-(\mu_{\alpha})^i_{\b{j}}\o{(\mu_{\beta})^s_{\b{t}}}R_{i\b{l}}\phi^{\b{l}k}\phi^{\b{j}t}\phi_{k\b{s}}.
\end{split}	
\end{align*}
 By using Stoke's theorem again and noticing $\n_k\phi_{i\b{j}}=0$, we have
 \begin{align*}
 	&\quad \int_{\mc{X}/M}(\Delta\rho) c(\phi)\frac{\omega^n}{n!}=\int_{\mc{X}/M}\phi^{\b{l}s}\p_s\p_{\b{l}}(R_{k\b{t}}\phi^{\b{t}k})c(\phi)\frac{\omega^n}{n!}\\
 	&=\int_{\mc{X}/M}\phi^{\b{l}s}(\n_{\b{l}}\n_sR_{k\b{t}})\phi^{\b{t}k}c(\phi)\frac{\omega^n}{n!}\\
 	&=\int_{\mc{X}/M}\phi^{\b{l}s}(\n_{\b{l}}\n_k R_{s\b{t}})\phi^{\b{t}k}c(\phi)\frac{\omega^n}{n!}\\
 	&=\int_{\mc{X}/M}\phi^{\b{l}s}R_{s\b{t}}\phi^{\b{t}k}\n_{\b{l}}\n_kc(\phi)\frac{\omega^n}{n!}\\
 	&=\int_{\mc{X}/M}\phi^{\b{l}s}R_{s\b{t}}\phi^{\b{t}k}\p_k\p_{\b{l}}c(\phi)_{\alpha\b{\beta}}\frac{\omega^n}{n!}\sqrt{-1}dz^{\alpha}\wedge d\b{z}^{\beta}\\
 	&=\int_{\mc{X}/M}(-R^i_{j\alpha\b{\beta}}R^j_{ik\b{l}}\phi^{k\b{l}}-(\mu_{\alpha})^i_{\b{j}}\o{(\mu_{\beta})^s_{\b{t}}}R_{i\b{l}}\phi^{\b{l}k}\phi^{\b{j}t}\phi_{k\b{s}})\frac{\omega^n}{n!}\sqrt{-1}dz^{\alpha}\wedge d\b{z}^{\beta}\\
 	&=\left(-\int_{\mc{X}/M}(R^i_{j\alpha\b{\beta}}R^j_{ik\b{l}}\phi^{k\b{l}})\frac{\omega^n}{n!}-\langle \mu_{\alpha}, \mu_{\beta}\rangle_{Ric}\right)\sqrt{-1}dz^{\alpha}\wedge d\b{z}^{\beta}, 
 \end{align*}
which proves (\ref{le2}).
\end{proof}
The first term in the RHS of (\ref{Bisexp}) is, by (\ref{omega}) and Lemma \ref{lemma0},
\begin{align}\label{2.50}
\begin{split}
\int_{\mc{X}/M}\frac{\omega^{n+1}}{(n+1)!}
&=\int_{\mc{X}/M}\frac{(c(\phi)+\sqrt{-1}\phi_{i\b{j}}\delta v^i\wedge \delta v^j)^{n+1}}{(n+1)!}\\
&=\int_{\mc{X}/M}\frac{(n+1)c(\phi)(\sqrt{-1}\phi_{i\b{j}}\delta v^i\wedge \delta v^j)^{n}}{(n+1)!}\\
&=\int_{\mc{X}/M}c(\phi)\frac{\omega^n}{n!}.
\end{split}
\end{align}

We evaluate the 
Quillen curvature (\ref{Bisexp})
at  the vector $\zeta=\zeta^{\alpha}\frac{\p}{\p z^{\alpha}}\in T_yM$.
It is, by (\ref{2.37}), (\ref{2.39}), (\ref{le1}), (\ref{le2}), (\ref{le4}) and (\ref{2.50}),
\begin{align}\label{Bis1}
\begin{split}
	&\quad(\sqrt{-1}c_1(\lambda,\|\bullet\|))(\zeta,\b{\zeta})\\
	&=\frac{k^{n+1}}{(2\pi)^{n+1}}\int_{\mc{X}_y}(-\sqrt{-1})c(\phi)(\zeta,\b{\zeta})\frac{\omega^n}{n!}+\frac{k^n}{(2\pi)^{n+1}}\int_{\mc{X}_y}\left(\frac{1}{2}|\mu|^2-\frac{\rho}{2}(-\sqrt{-1})c(\phi)(\zeta,\b{\zeta})\right)\frac{\omega^n}{n!}\\
	&\quad+(-\sqrt{-1})\frac{k^{n-1}}{(2\pi)^{n+1}}\int_{\mc{X}_y}\frac{1}{8}(\sqrt{-1}R^{K_{\mc{X}/M}})^2\wedge\frac{\omega^{n-1}}{(n-1)!}(\zeta,\b{\zeta})\\
	&\quad -\frac{k^{n-1}}{(2\pi)^{n+1}}\frac{1}{24}\int_{\mc{X}_y}\left((-\sqrt{-1})c(\phi)(\zeta,\b{\zeta})(-2\Delta \rho+|Ric|^2-|R|^2)\right)\frac{\omega^n}{n!}\\
	&\quad+\frac{k^{n-1}}{(2\pi)^{n+1}}\frac{1}{12}(\|\n'\mu\|^2+\|\mu\|^2_{Ric})+O(k^{n-2}).
\end{split}	
\end{align}

Note that
\begin{align}\label{2.51}
\begin{split}
&\quad R^{K_{\mc{X}/M}}=\p\b{\p}\log\det\phi\\
	&=(\p\b{\p}\log\det\phi)(\frac{\delta}{\delta z^{\alpha}},\frac{\delta}{\delta \b{z}^{\beta}})dz^{\alpha}\wedge d\b{z}^{\beta}+(\p\b{\p}\log\det\phi)(\frac{\delta}{\delta z^{\alpha}},\frac{\p}{\p \b{v}^{j}})dz^{\alpha}\wedge \delta\b{v}^{j}\\
	&\quad+(\p\b{\p}\log\det\phi)(\frac{\p}{\p v^{i}},\frac{\delta}{\delta \b{z}^{\beta}})\delta v^{i}\wedge d\b{z}^{\beta}+(\p\b{\p}\log\det\phi)(\frac{\p}{\p v^{i}},\frac{\p}{\p \b{v}^{j}})\delta v^{i}\wedge \delta\b{v}^{j}
	\end{split}
\end{align}
and 
\begin{align}\label{2.52}
\begin{split}
	\b{\p}^*\mu_{\alpha}=(\p\b{\p}\log \det\phi)(\frac{\delta}{\delta z^{\alpha}},\frac{\p}{\p\b{v}^j})\phi^{\b{j}i}\frac{\p}{\p v^i}.
\end{split}	
\end{align}
In fact, one can prove (\ref{2.52}) in terms of normal coordinates, at a fixed point, one has
\begin{align*}
\b{\p}^*\mu_{\alpha} &=-\sqrt{-1}[\Lambda,\n']((\mu_\alpha)^i_{\b{j}}d\b{v}^j\otimes \frac{\p}{\p v^i})\\
&=-\phi^{s\b{j}}(\p_s(\mu_{\alpha})^i_{\b{j}}+\mu^k_{\b{j}}\Gamma^i_{ks})\frac{\p}{\p v^i}\\
&=\p_j\p_{\b{j}}(\phi_{\alpha\b{l}}\phi^{\b{l}i})\frac{\p}{\p v^i}\\
&=\left(\phi_{\alpha\b{i}j\b{j}}-\phi_{\alpha\b{l}}\phi_{j\b{j}l\b{i}}\right)\frac{\p}{\p v^i}\\
&=(\p\b{\p}\log \det\phi)(\frac{\delta}{\delta z^{\alpha}},\frac{\p}{\p\b{v}^j})\phi^{\b{j}i}\frac{\p}{\p v^i}.
\end{align*}

By (\ref{2.51}), (\ref{2.52}), (\ref{2.53}) and (\ref{2.54}), 
{the integral in the third  term in the RHS of (\ref{Bis1})
can be computed as
\begin{align*}
	\begin{split}
		&\quad\int_{\mc{X}/M}(\sqrt{-1}R^{K_{\mc{X}/M}})^2\wedge \frac{\omega^{n-1}}{(n-1)!}\\
		&=2\int_{\mc{X}/M}(\p\b{\p}\log\det\phi)(\frac{\delta}{\delta z^{\alpha}},\frac{\delta}{\delta \b{z}^{\beta}})(\p\b{\p}\log\det\phi)(\frac{\p}{\p v^{i}},\frac{\p}{\p \b{v}^{j}})\sqrt{-1}\delta v^{i}\wedge\delta\b{v}^{j}\frac{\omega^{n-1}}{(n-1)!}\sqrt{-1}dz^{\alpha}\wedge d\b{z}^{\beta}\\
		&-2\int_{\mc{X}/M}(\p\b{\p}\log\det\phi)(\frac{\delta}{\delta z^{\alpha}},\frac{\p}{\p \b{v}^{j}})(\p\b{\p}\log\det\phi)(\frac{\p}{\p v^{i}},\frac{\delta}{\delta \b{z}^{\beta}}) \sqrt{-1}\delta v^{i}\wedge\delta\b{v}^{j}\frac{\omega^{n-1}}{(n-1)!}\sqrt{-1}dz^{\alpha}\wedge d\b{z}^{\beta}\\
		&+\int_{\mc{X}/M}(\p\b{\p}\log\det\phi)(\frac{\p}{\p v^{i}},\frac{\p}{\p \b{v}^{j}})(\p\b{\p}\log\det\phi)(\frac{\p}{\p v^{k}},\frac{\p}{\p \b{v}^{l}})\sqrt{-1}\delta v^{i}\wedge \delta\b{v}^{j}\wedge \sqrt{-1}\delta v^{k}\wedge \delta\b{v}^{l}c(\phi)\frac{\omega^{n-2}}{(n-2)!}\\
		&=-2\int_{\mc{X}/M}\rho(\Delta c(\phi)_{\alpha\b{\beta}}+(\mu_{\alpha})^i_{\b{j}}\o{(\mu_{\beta})^t_{\b{s}}}\phi^{\b{j}s}\phi_{i\b{t}})\frac{\omega^n}{n!}\sqrt{-1}dz^{\alpha}\wedge d\b{z}^{\beta}\\
		&\quad -2\langle\b{\p}^*\mu_{\alpha},\b{\p}^*\mu_{\beta}\rangle\sqrt{-1}dz^{\alpha}\wedge d\b{z}^{\beta}+\int_{\mc{X}/M}(\rho^2-|Ric|^2)c(\phi)\frac{\omega^n}{n!}.
		\end{split}
\end{align*}
Evaluted at the vector $\zeta=\zeta^{\alpha}\frac{\p}{\p z^{\alpha}}\in T_yM$ it is,  by Stoke's theorem,
\begin{align}\label{2.55}
\begin{split}
	&\quad(-\sqrt{-1})\int_{\mc{X}/M}(\sqrt{-1}R^{K_{\mc{X}/M}})^2\wedge \frac{\omega^{n-1}}{(n-1)!}(\zeta,\b{\zeta})\\
	&=\int_{\mc{X}_y}\left((-\sqrt{-1})c(\phi)(\zeta,\b{\zeta})(-2\Delta\rho+\rho^2-|Ric|^2)-2\rho|\mu|^2\right)\frac{\omega^n}{n!}-2\|\b{\p}^*\mu\|^2.
	\end{split}
\end{align}

Substituting (\ref{2.55}) into (\ref{Bis1}) we have proved 
\begin{prop} 
The  curvature of the Quillen metric has the following expansion in $k$,
\begin{align}\label{asymptotic 2}
\begin{split}
&\quad (\sqrt{-1}c_1(\lambda,\|\bullet\|))(\zeta,\b{\zeta})\\
&=\frac{k^{n+1}}{(2\pi)^{n+1}}\int_{\mc{X}_y}(-\sqrt{-1})c(\phi)(\zeta,\b{\zeta})\frac{\omega^n}{n!}+\frac{k^n}{(2\pi)^{n+1}}\int_{\mc{X}_y}\left(\frac{1}{2}|\mu|^2-\frac{\rho}{2}(-\sqrt{-1})c(\phi)(\zeta,\b{\zeta})\right)\frac{\omega^n}{n!}\\
	&\quad+\frac{k^{n-1}}{(2\pi)^{n+1}}\int_{\mc{X}_y}\left((-\sqrt{-1})c(\phi)(\zeta,\b{\zeta})\left(-\frac{1}{6}\Delta\rho+\frac{1}{24}(|R|^2-4|Ric|^2+3\rho^2)\right)-\frac{\rho}{4}|\mu|^2\right)\frac{\omega^n}{n!}\\
	&\quad +\frac{k^{n-1}}{(2\pi)^{n+1}}\left(\frac{1}{12}\|\mu\|^2_{Ric}+\frac{1}{12}\|\n'\mu\|^2-\frac{1}{4}\|\b{\p}^*\mu\|^2\right)+O(k^{n-2}).
\end{split}
\end{align}
\end{prop} 

From above Proposition, we proved Theorem \ref{main theorem 2}.

\subsection{An application}

In this subsection, we will prove Corollary \ref{cor1}.

For any positive integer $k$ write temporarily
\begin{align*}
F=L^k+K_{\mc{X}/M},
\end{align*}
where $K_{\mc{X}/M}$ is the relative canonical line bundle  endowed with the following metric,
\begin{align}\label{2.321}
(\det\phi)^{-1}:=(\det(\phi_{i\b{j}}))^{-1}.	
\end{align}

As in subsection \ref{sub1}, the operator $D_y=\b{\p}_y+\b{\p}^*_y$ acts on  $\oplus_{p\geq 0}A^{0,p}(\mc{X}_y, F)$. 
Take  a  small constant $b>0$ that is smaller than the all positive eigenvalues of $D_y$. Then
\begin{align*}
K^{b,p}_y\cong H^{p}(\mc{X}_y,F)=H^p(\mc{X}_y, K_{\mc{X}_y}+L^k).
\end{align*}
 By Kodaira vanishing theorem,
  \begin{align*}
  K^{b,0}_y\cong H^{0}(\mc{X}_y,L^k+K_{\mc{X}_y})\cong \pi_*(L^k+K_{\mc{X}/M})_y  \quad K_y^{b,p}=0, \quad \text{for}\quad p\geq 1. 	
  \end{align*}
  So
  \begin{align}\label{3.1}
  \lambda^b=(\det \pi_*(L^k+K_{\mc{X}_y}))^{-1}.	
  \end{align}

By (\ref{L2 metric}) and (\ref{2.321}), we have
\begin{align}\label{metric consist}
|u|^2e^{-k\phi}=(\sqrt{-1})^2|u'|^2e^{-k\phi}dv\wedge d\b{v}=|u'|^2e^{-k\phi}(\det\phi)^{-1}\frac{\omega^n}{n!}=|u|^2_{L^2}\frac{\omega^n}{n!},
\end{align}
that is, the $L^2$-metric $\|\bullet\|_k$ on $\pi_*(L^k+K_{\mc{X}/M})$ given by (\ref{L2 metric}) coincides with the standard $L^2$-metric on $\pi_*(L^k+K_{\mc{X}/M})$ induced by $(\mc{X}_y, \omega|_y)$, $(K_{\mc{X}_y}, (\det\phi)^{-1})$ and $(L,e^{-\phi})$.
Therefore, the $L^2$-metric $(|\bullet|^b)^2$ is dual to the determinant of the metric $\|\bullet\|^2$.  By (\ref{torsion}), we have
\begin{align}\label{torsion1}
(\|\bullet\|^b)^2=(|\bullet|^b)^2(\tau_k(\b{\p}^{(b,+\infty)}))^2=(\det\|\bullet\|^2_k)^* (\tau_k(\b{\p}))^2,	
\end{align}
for $b>0$ small enough, where $\tau_k(\b{\p})=\tau_{k}(\b{\p}^{(b,+\infty)})$ is the analytic torsion associated with $(\mc{X}, \omega=\sqrt{-1}\p\b{\p}\phi))$ and $(L^k, e^{-k\phi})$. Therefore,
\begin{align}\label{Torsion}
\begin{split}
\frac{\sqrt{-1}}{2\pi}\p\b{\p}\log (\tau_k(\b{\p}))^2
&=-c_1(\lambda,\|\bullet\|_Q)-c_1(E^k, \|\bullet\|_k).
\end{split}
\end{align}

\vspace{5mm}

{\it Proof of Corollary \ref{cor1}.}
Substituting (\ref{asymptotic 1}) and (\ref{asymptotic 2}) into (\ref{Torsion}), we obtain
\begin{align*}
\begin{split}
\frac{1}{2\pi}\p\b{\p}\log(\tau_k(\b{\p}))^2(\zeta,\b{\zeta})
&=(-\sqrt{-1})(-c_1(\lambda,\|\bullet\|_Q)-c_1(E^k, \|\bullet\|_k))(\zeta,\b{\zeta})\\
&=o\left(k^{n-1}\right).\\
\end{split}
\end{align*}
Therefore, 
\begin{align*}
\p\b{\p}\log\tau_k(\b{\p})=o(k^{n-1}).	
\end{align*}
 
\rightline{$\Box$}

\end{document}